\documentclass[11pt]{amsart}
\usepackage[T1]{fontenc}
\usepackage{lmodern}
\usepackage{amsmath,amssymb,mathtools,mathrsfs}
\usepackage{enumitem}
\usepackage[margin=1.05in]{geometry}
\usepackage[hidelinks]{hyperref}
\newcommand{\A}{\mathbb A}
\newcommand{\C}{\mathbb C}
\newcommand{\Gm}{\mathbb G_m}

\newcommand{\GW}{\operatorname{GW}}
\newcommand{\Spec}{\operatorname{Spec}}

\newcommand{\Hom}{\operatorname{Hom}}
\newcommand{\Map}{\operatorname{Map}}

\newcommand{\Sm}{\operatorname{Sm}}
\newcommand{\Nis}{\mathrm{Nis}}
\newcommand{\id}{\operatorname{id}}

\newcommand{\Frac}{\operatorname{Frac}}
\newcommand{\MW}{\mathrm{MW}}
\newcommand{\linecomplement}{\mathbb A^3_{\mathbb C}\setminus i(\mathbb A^1_{\mathbb C})}

\newtheorem{theorem}{Theorem}[section]
\newtheorem{proposition}[theorem]{Proposition}
\newtheorem{lemma}[theorem]{Lemma}
\newtheorem{corollary}[theorem]{Corollary}

\newtheorem{question}[theorem]{Question}
\theoremstyle{definition}
\newtheorem{definition}[theorem]{Definition}
\newtheorem{example}[theorem]{Example}

\theoremstyle{remark}
\newtheorem{remark}[theorem]{Remark}
\numberwithin{equation}{section}
\setlist[enumerate]{label=\textup{(\roman*)},leftmargin=*}
\title[Motivic rectifiability and concrete rectifications]{On the motivic approach to the Abhyankar rectifiability problem and concrete rectifications}
\author{A.\,S. Frolov}
\address{Higher School of Modern Mathematics, MIPT}
\email[A.\,S. Frolov]{frolov.aleksandr@phystech.edu}
\author{I.\,S. Dorofeev}
\address{Independent researcher}
\email[I.\,S. Dorofeev]{qft98romanov@gmail.com}
\date{October 27, 2026}

\newcommand{\Z}{\mathbb Z}

\newcommand{\susp}{\Sigma_s}
\newcommand{\piA}[1]{\pi_{#1}^{\mathbb A^1}}
\newcommand{\HA}[1]{\mathbf H_{#1}^{\mathbb A^1}}
\DeclareMathOperator{\ev}{ev}
\DeclareMathOperator{\Tors}{Tors}

\hypersetup{pdftitle={On the motivic approach to the Abhyankar rectifiability problem and concrete rectifications},pdfauthor={A. S. Frolov and I. S. Dorofeev}}

\begin{document}
\begin{abstract}
We study closed embeddings of the complex affine line into affine
three-space. We prove that,
for every such embedding, the complete-intersection normal map
identifies the $\mathbb A^1$-fundamental sheaf of its complement
with $\mathbf K_2^{\mathrm{MW}}$. The proof compares torsors under
arbitrary strongly $\mathbb A^1$-invariant group sheaves. A nodal
projection expresses a torsor by a normal parameter and finitely many
node parameters. Extension across the node axes gives equations for
the latter. Their correction terms increase a nilpotent filtration
generated by the Milnor--Witt symbols of the branch-height differences.
We deduce the vanishing of
the first Alexander sheaf of every complex embedded line.
We also construct polynomial rectifications using finite inverses and
affine modifications, including rectifications of the Shastri
trefoil and a harmonic figure-eight over
$\mathbb Z[a]/(a^2-a+1)$, with arbitrary base change.
The companion repository contains explicit rectifiers, their inverses,
proofs, and symbolic checks:
\url{https://github.com/qft98romanov-hash/abhyankar-moh-trefoil}.
We prove that a normal map with a polynomial section and geometric
generic affine-line fiber is a coordinate projection. We formulate
the remaining rectifiability and descent questions; the general complex
rectifiability problem is not settled here. An appendix treats other
fields, normal orientations, and arithmetic descent.
\end{abstract}
\maketitle
\tableofcontents
\setcounter{section}{-1}
\section{Introduction}\label{sec:introduction}

In this paper we study closed embeddings of the affine line into
complex affine three-space. The Abhyankar rectifiability problem in
dimensions $(1,3)$ asks whether every such embedding can be carried
to a coordinate line by a polynomial automorphism. Equivalently,
given polynomials $p,q,r\in\C[t]$ with
$\C[p,q,r]=\C[t]$, one asks for a polynomial automorphism
$\Phi$ satisfying
\[
 \Phi\bigl(p(t),q(t),r(t)\bigr)=(t,0,0).
\]
The analogous problem in the affine plane is settled by the
Abhyankar--Moh--Suzuki theorem \cite{AM,Suzuki}.

Our first result computes the unstable motivic fundamental sheaf
of the complement. An embedded affine line has a trivial rank-two
normal bundle, so its local motivic meridian is the punctured affine
plane. The latter has fundamental sheaf $\mathbf K_2^{\MW}$,
the unramified Milnor--Witt $K$-theory sheaf. Homotopy purity
identifies the corresponding suspended complements, but suspension
does not retain all nonabelian information. The main issue is
therefore to show that no additional torsor data survive away
from a normal meridian.

\begin{samepage}
\begin{theorem}
\label{thm:intro-fundamental}
Let $i:\A^1_{\C}\hookrightarrow\A^3_{\C}$ be a closed embedding.
Choose generators $f,g$ of its defining ideal and compatible
rational basepoints. The normal map induces an isomorphism
of Nisnevich sheaves on $\Sm_{\C}$:
\[
 \pi_1^{\A^1}\!\left(\A^3_{\C}\setminus i(\A^1_{\C})\right)
 \xrightarrow[\ (f,g)_*\ ]{\ \sim\ }
 \pi_1^{\A^1}(\A^2_{\C}\setminus\{0\})
 \cong\mathbf K_2^{\MW}.
\]
\end{theorem}
\end{samepage}

The proof uses torsors under arbitrary strongly $\A^1$-invariant
group sheaves. A
generic finite projection of the curve to a plane has finitely
many ordinary nodes. Away from the nodes, an affine-line torsor
allows one to translate the height of the embedded curve to zero.
The remaining clutching data have one normal coefficient and
finitely many additional coefficients associated with the nodes.
Extension across the vertical axis over each node imposes an equation
on those additional coefficients.

There are two essential features of the proof. First, specialization
must retain the full nonabelian chart gauges. Keeping only an
abelian residue would lose precisely the information being
computed. Second, the heights of the two branches over a node
need not differ by $1$. Their actual nonzero difference occurs
as a constant Milnor--Witt symbol in the node equation. We compare
an arbitrary torsor with the pullback of its normal-meridian
torsor, keeping these constants. The resulting finite system
strictly raises a filtration generated by the Milnor--Witt symbols
of the branch-height differences. If the projection has $r$ nodes, its $(r+1)$st level
is zero: the constant symbols anticommute and each has square
zero over $\C$. The additional constants in the chart coordinates
preserve this filtration and do not have to be included among its
generators. Thus all the extra coefficients vanish after at most
$r+1$ substitutions in the finite system.

Wendt and Wittich construct a universal abelian
$\mathbf K_2^{\MW}$-covering for embedded-line complements and
develop the associated Alexander-type invariant \cite{WW26}.
To pass from their abelian comparison to
Theorem~\ref{thm:intro-fundamental}, we prove the nonabelian torsor
comparison described above. It also determines their first
Alexander sheaf over $\C$.

\begin{corollary}
\label{cor:intro-alexander}
For every closed embedding $i:\A^1_{\C}\hookrightarrow\A^3_{\C}$,
the universal abelian $\mathbf K_2^{\MW}$-covering of Wendt and
Wittich is universal. Its total space satisfies
\begin{align*}
 \pi_1^{\A^1}\!\left(
   \widetilde{\A^3_{\C}\setminus i(\A^1_{\C})}\right)&=0,\\
 \mathbf H_1^{\A^1}\!\left(
   \widetilde{\A^3_{\C}\setminus i(\A^1_{\C})}\right)&=0.
\end{align*}
In particular, their first Alexander sheaf vanishes for every
complex embedded line.
\end{corollary}

\begin{proof}
The normal map is an isomorphism on the fundamental sheaf by
Theorem~\ref{thm:intro-fundamental}, and on first homology by
purity. Naturality of Hurewicz therefore makes the Hurewicz map
of the complement an isomorphism. The universal abelian covering
has fundamental sheaf equal to its kernel
\cite[Definition~4.4 and Theorem~4.8]{WW26}; hence that kernel is
zero. Its total space is $\A^1$-connected, so its first
$\A^1$-homology also vanishes.
\end{proof}

There is a second part of the knot-theoretic comparison. Classically
the first homology of the infinite cyclic cover of a knot exterior,
with the action of its deck group, is the Alexander module.
Wendt and Wittich replace that cyclic cover by the universal
abelian $\mathbf K_2^{\mathrm{MW}}$-covering and consider its
first $\A^1$-homology sheaf \cite[Definition~4.11]{WW26}.

In the second part we construct polynomial rectifications using
finite inverses and affine modifications.
An automorphism of a finite
quotient modulo a power of a parameter can be lifted to a
polynomial automorphism after adding one suitable coordinate.
Exact divisibility proves polynomiality along the exceptional
divisor, where an inverse written only after localization would
be insufficient.

\begin{theorem}
\label{thm:intro-concrete}
Put $\mathcal R=\mathbb Z[a]/(a^2-a+1)$.
The Shastri parametrization
\[
 t\longmapsto(t^3-3t,\ t^4-4t^2,\ t^5-10t)
\]
and the harmonic Chebyshev parametrization
\[
 t\longmapsto
 \bigl(t^3-3t,\ t^5-5t^3+5t,\
              t^7-7t^5+14t^3-7t\bigr)
\]
admit polynomial rectifiers and polynomial inverses over $\mathcal R$.
These identities remain valid after every base change from $\mathcal R$.
In particular, both parametrizations are rectifiable over every
field containing a root of $a^2-a+1$, including in
characteristics two and three, and hence over $\C$.
\end{theorem}

The second part gives the rectifiers and verifies their inverse identities;
reproducible symbolic checks are available in \cite{DorofeevRepository}.

\subsection{Further examples}
\label{subsec:intro-unresolved}

We have not obtained polynomial rectifiers over $\C$ for the
following embeddings:
\begin{align}
 t&\longmapsto
 \bigl(t^3-3t,\ t^5-5t^3+4t,\
             t^7-7t^5+14t^3-7t\bigr),
 \label{eq:intro-unresolved-chebyshev}\\
 t&\longmapsto\bigl(t^3,\ t^5,\ t+t^2+2t^4\bigr),
 \label{eq:intro-unresolved-monomial}\\
 t&\longmapsto\bigl(t^n,\ t^{n+1},\ t+t^{n+2}\bigr),
       \qquad n\geq5.
 \label{eq:intro-unresolved-successive}
\end{align}

The coefficient $4t$ in
\eqref{eq:intro-unresolved-chebyshev} distinguishes this case from
the harmonic embedding with $5t$, which is rectified in
Theorem~\ref{thm:harmonic-rectification}. The first unresolved
case also has the presentation
\[
 t\longmapsto
 (t^3-3t,\ t^5-5t^3+4t,\ t^7-42t):
\]
the affine ambient shear $(x,y,z)\mapsto(x,y,z-7y-21x)$
takes this parametrization to
\eqref{eq:intro-unresolved-chebyshev}.

Each map in \eqref{eq:intro-unresolved-chebyshev}--
\eqref{eq:intro-unresolved-successive} is a closed embedding.
Parameter-recovery polynomials, in the same order, are
\begin{align*}
 P_4(x,y,z)={}&-x(y-x)^2-3x^2(y-x)+4(y-x)\\
             &+(x^2-3)(z+y-4x),\\
 P_{3,5}(x,y,z)={}&z-z^2+2x+5y+2x^2+4xy-xz+2yz-4x^3,\\
 P_n(x,y,z)={}&z\sum_{j=0}^{n-1}(-y)^j+(-1)^n x^ny.
\end{align*}
Their restrictions to the corresponding curves are $t$.
For the last identity, use
$z=t(1+y)$ and $x^ny=t y^n$, followed by the finite geometric
sum. These polynomials recover the parameter; they do not constitute
rectifications.

\subsection{Comparison with knot theory}

For a smooth oriented knot $K\subset S^3$, choose a closed
tubular neighborhood $\nu K\cong S^1\times D^2$. A meridian
on its boundary torus bounds a normal disk in $\nu K$; the
preferred longitude is a parallel copy of $K$ with linking
number zero. Every exterior
$S^3\setminus\operatorname{int}\nu K$ has first homology
$\mathbb Z$, generated by a meridian. The fundamental group
detects the unknot: $K$ is the unknot precisely when
\[
 \pi_1\bigl(S^3\setminus\operatorname{int}\nu K\bigr)
       \cong\mathbb Z
\]
\cite[Theorem~3]{Pap57}. The preferred longitude on the boundary
torus is homologically zero. If the group is $\mathbb Z$,
that longitude is null-homotopic, and Dehn's lemma gives a
properly embedded disk in the exterior with this boundary.
Joining it to the annulus between the longitude and the knot
produces a disk in $S^3$ bounding $K$. The longitude is the class
filled by Dehn's lemma; the meridian remains nonzero in homology,
even for the unknot.

For a complex embedded affine line, the normal puncture is
$\A^2_{\C}\setminus\{0\}$ and its complex normal link is
$S^3$, because the real codimension is four. This link is simply
connected, whereas the motivic puncture has fundamental sheaf
$\mathbf K_2^{\mathrm{MW}}$. A framing
identifies the punctured normal bundle with
\[
 N_{i(\A^1_{\C})/\A^3_{\C}}^{\times}
       \cong\A^1_{\C}\times(\A^2_{\C}\setminus\{0\}),
\]
where the zero section is removed. Because the affine-line factor is
$\A^1$-contractible, there is no second fundamental-group factor
corresponding to the longitude on a boundary torus. This normal model
does not by itself define an algebraic tubular map into the complement.
An analogue of the embedded-disk step would still be needed to obtain
ambient polynomial coordinates.

The real polynomial examples further clarify the analogy.
A nontrivial real long knot cannot be rectified by a real
polynomial automorphism, whereas an automorphism over $\C$
need not preserve the real locus. The explicit rectifications
above therefore coexist with the real knot types of their
parametrizations.

\subsection{Further questions}

\begin{question}\label{qu:intro-main}
Does it follow from Theorem~\ref{thm:intro-fundamental} that every
closed embedding $i:\A^1_{\C}\hookrightarrow\A^3_{\C}$ is
rectifiable over $\C$? More explicitly, can the common formula
\[
 \pi_1^{\A^1}\!\left(\A^3_{\C}\setminus i(\A^1_{\C})\right)
       \cong\mathbf K_2^{\MW}
\]
and the actual normal comparisons used to prove it be strengthened
to produce, for every $i$, a polynomial automorphism $\Phi$ of
$\A^3_{\C}$ such that $\Phi(i(t))=(t,0,0)$?
\end{question}

Question~\ref{qu:intro-main} remains open in this manuscript.
An equivalent geometric formulation asks for generators $f,g$ of
the curve ideal such that $(f,g):\A^3\to\A^2$ has a polynomial
section and generic fiber $\A^1_{\C(f,g)}$.
Corollary~\ref{cor:normal-projection-rectification} constructs a
rectifier from these data by successive exact divisions. Conversely,
a rectifier with coordinates $(T,f,g)$ supplies the section by its
inverse at $(0,u,v)$.

Proposition~\ref{prop:bounded-inverse-jets} likewise identifies
rectifiability with the existence of inverse polynomial families of
bounded spatial degree along the contraction to the tangent line.
A quantitative solution would give explicit bounds, in terms of the
component degrees of the embedding, for the degrees of both a rectifier
and its inverse.

Over infinite perfect fields of characteristic different from two,
Proposition~\ref{prop:other-field-homology} reduces the normal-map
comparison to abelianity of the fundamental sheaf.

\begin{question}\label{qu:intro-fields}
For which infinite perfect fields $k$ of characteristic different
from two is
\[
 \pi_1^{\A^1}\!\left(\A^3_k\setminus i(\A^1_k)\right)
\]
abelian for every closed embedding $i:\A^1_k\hookrightarrow\A^3_k$?
Can the node argument be extended when projection nodes do not split
over $k$, or when constant Milnor--Witt symbols retain nontrivial
quadratic information?
\end{question}

The intrinsic characteristic-two comparison is stated separately in
Question~\ref{qu:other-characteristic-two}.

Theorem~\ref{thm:coordinate-transfer} tests a specified coordinate
and modification. Applying it to a given curve requires choosing these
data, possibly through several transfers.

\begin{question}\label{qu:intro-families}
Can a polynomial rectifier for
\eqref{eq:intro-unresolved-chebyshev}, whose second component contains
$+4t$, be constructed by repeated applications of
Theorem~\ref{thm:coordinate-transfer} and affine coordinate changes?
More generally, which closed Chebyshev parametrizations admit a finite
chain of these transfers starting from the rectified models in
Section~\ref{sec:rectifications}?
\end{question}

The displayed formulas give both rectifications over
$\mathbb Q(\sqrt{-3})$. The real-knot obstruction excludes
$\mathbb Q$, but gives no answer over other imaginary quadratic
fields. The next question tests descent to a field which has no real
embeddings and does not contain the coefficients used in our formulas.

\begin{question}\label{qu:intro-descent}
Do the two parametrizations with rational coefficients in
Theorem~\ref{thm:intro-concrete} admit polynomial rectifiers over
$\mathbb Q(\sqrt{-1})$?
\end{question}

\subsection{Organization of the paper}

Section~\ref{sec:reminders} recalls the motivic constructions
used in the proofs. The first part develops contractions,
specialization, the finite-node geometry, and the full torsor
comparison proving Theorem~\ref{thm:intro-fundamental}. The
second part treats polynomial rectification constructions and
the explicit examples of Theorem~\ref{thm:intro-concrete}.
The subsequent section studies the proposed geometric route to
universal rectifiability and states its missing steps as questions.
The appendix separates the results available over other fields
from the specifically complex ingredients of the proof.

\section*{Acknowledgements}

We thank Professor Alexey Kanel-Belov for drawing our attention
to the rectifiability problem and for suggesting the use of
knot-theoretic intuition. We thank Timofey Fedorov and Savva
Shatov for useful discussions. The second author expresses deep
gratitude to Professor Alexander Gasnikov. Both authors are members of the
``Algebra, Geometry and Number Theory'' seminar
[\url{https://agnt.thematdev.org}] at the Higher School
of Modern Mathematics, Moscow Institute of Physics and Technology.

\section*{Disclosure of AI assistance}

Generative AI tools, including ChatGPT Sol 6.1, were used in developing the research notes
and preparing this manuscript, including literature searches,
mathematical exploration, symbolic computations, organization,
drafting, and revision. Some arguments were developed with AI
assistance. The named authors retain responsibility for the
mathematical content, the interpretation of the cited sources,
and the verification required before submission.

\part{The motivic fundamental sheaf}
\section{Definitions and reminders on motivic invariants}\label{sec:reminders}

We recall the motivic constructions used in the first part. The
fundamental-group invariant is a Nisnevich sheaf whose sections vary
over smooth schemes. Throughout the main text the ground field is $\C$.

\subsection{The Nisnevich topology and \texorpdfstring{$\A^1$}{A1}-homotopy}

The relevant topology is designed to retain residue fields when
passing to local coordinates. This matters in a calculation involving
normal directions and local degrees.

\begin{definition}
A family of \'etale morphisms $\{X_j\to X\}$ is a
\emph{Nisnevich covering} if every point $x\in X$ has a lift
$x_j\in X_j$ for which the induced residue-field map
$\kappa(x)\to\kappa(x_j)$ is an isomorphism. We write $\Sm_k$ for
the category of smooth schemes of finite type over $k$, equipped with
this topology.
\end{definition}

One starts with simplicial Nisnevich sheaves on $\Sm_k$ and formally
makes each projection $X\times\A^1\to X$ a weak equivalence.
The resulting homotopy theory is the unstable motivic homotopy theory
of Morel and Voevodsky \cite{MV99}. Equivalence here is taken in this
localized category; it need not be represented by a polynomial
isomorphism between the original schemes.

For a pointed space $(X,x)$, its fundamental sheaf
$\pi_1^{\A^1}(X,x)$ is the Nisnevich sheaf associated with the
presheaf of pointed maps $S_s^1\wedge T_+\to X$ in this category.
Here $S_s^1$ denotes the simplicial circle, and $T$ ranges over
$\Sm_k$. The superscript $\A^1$ therefore changes both the maps
and the equivalence relation used to define a fundamental group.
It does not mean the topological fundamental group of the complex
points.

The complements considered here are $\A^1$-connected by
Lemma~\ref{geom:connected}. We fix rational basepoints and compatible
paths; other choices change induced homomorphisms by inner automorphisms.

\subsection{Strong invariance, strict invariance, and torsors}

Maps to classifying spaces provide a way to test a fundamental sheaf
without choosing presentations for it. This is useful because the
fundamental sheaf of an arbitrary complement is not assumed to be
abelian at the outset.

\begin{definition}
A Nisnevich sheaf of groups $\mathcal Q$ is \emph{strongly
$\A^1$-invariant} if pullback induces bijections
\[
 H^j_{\Nis}(X,\mathcal Q)\xrightarrow{\sim}
 H^j_{\Nis}(X\times\A^1,\mathcal Q),\qquad j=0,1,
\]
for every smooth $X$. The group $H^0$ is the group of sections;
$H^1$ is the pointed set of isomorphism classes of Nisnevich
$\mathcal Q$-torsors. An abelian sheaf $\mathcal M$ is
\emph{strictly $\A^1$-invariant} if this condition holds for every
$j\ge0$.
\end{definition}

Over a perfect field, an abelian strongly $\A^1$-invariant sheaf is
strictly $\A^1$-invariant \cite{Morel12}. For a strongly invariant
group sheaf, its classifying space $B\mathcal Q$ is $\A^1$-local.
Consequently a morphism into $B\mathcal Q$ records an actual
Nisnevich torsor up to isomorphism. For an $\A^1$-connected pointed
space, pointed morphisms into $B\mathcal Q$ correspond to
homomorphisms from its fundamental sheaf to $\mathcal Q$;
forgetting the pointing identifies conjugate homomorphisms.

We will also use unramifiedness. A strongly invariant group sheaf
on a smooth integral scheme embeds into its group of sections over
the function field, and a rational section extends if it extends
at every point of codimension one. In particular, deleting a closed
subset of codimension at least two does not change its sections.
It does not follow that every torsor on the smaller open extends.
What does follow is that an isomorphism between two torsors already
defined on the larger scheme extends uniquely: locally such an
isomorphism is a section of $\mathcal Q$, and uniqueness permits
the extensions to glue.

\subsection{Contractions and Milnor--Witt \texorpdfstring{$K$}{K}-theory}

The multiplicative group $\Gm=\A^1\setminus\{0\}$ is a second
basic circle in motivic homotopy theory. Evaluation at its basepoint
$1$ gives an algebraic loop construction for sheaves.

\begin{definition}
The \emph{contraction} of a strongly invariant group sheaf
$\mathcal Q$ is
\[
 \mathcal Q_{-1}(X)=
 \ker\bigl(\mathcal Q(X\times\Gm)
       \xrightarrow{\,t=1\,}\mathcal Q(X)\bigr).
\]
Further contractions are obtained by repeating this operation.
Thus a section of $\mathcal Q_{-n}$ is a family on $n$
multiplicative coordinates that becomes the identity if any one
of those coordinates is set equal to $1$.
\end{definition}

The notation $\mathbf K_n^{\MW}$ denotes the unramified
Milnor--Witt $K$-theory sheaf, not a constant sheaf. Its value on
a field $F$ is the degree-$n$ part of the graded ring generated
by $[a]$, for $a\in F^\times$, in degree $1$, and $\eta$ in
degree $-1$, with relations
\begin{align*}
 [ab]&=[a]+[b]+\eta[a][b], & [a][1-a]&=0\quad(a\ne1),\\
 \eta[a]&=[a]\eta, & \eta(2+\eta[-1])&=0.
\end{align*}
For smooth schemes, its sections are the unramified sections of
these groups at their function fields. Its degree-zero part is
the Grothendieck--Witt sheaf; the class $\langle a\rangle$
of a one-dimensional quadratic form equals $1+\eta[a]$.
The quotient by $\eta$ is ordinary Milnor $K$-theory.

The calculation of the punctured affine plane is the basic
sphere calculation used throughout:
\begin{equation}\label{eq:punctured-plane-reminder}
 \A^2_k\setminus\{0\}\simeq_{\A^1}
 S_s^1\wedge\Gm^{\wedge2},\qquad
 \pi_1^{\A^1}(\A^2_k\setminus\{0\})\cong\mathbf K_2^{\MW}.
\end{equation}
The second isomorphism is Morel's calculation \cite{Morel12}.
In particular, for every strongly invariant group sheaf
$\mathcal Q$, it gives the universal property
\begin{equation}\label{eq:square-universality-reminder}
 \Hom(\mathbf K_2^{\MW},\mathcal Q)
 \cong \mathcal Q_{-2}(k).
\end{equation}
This is valid for a possibly nonabelian target. One obtains it
by mapping the suspended square in
\eqref{eq:punctured-plane-reminder} into $B\mathcal Q$ and
then looping simplicially. No corresponding nonabelian universal
property for every higher Milnor--Witt sheaf is being asserted.

Over $\C$, the positive-degree groups $K_n^{\MW}(\C)$ agree
with $K_n^{\mathrm M}(\C)$. Two consequences that enter the
finite-node calculation are
\begin{equation}\label{eq:constant-squares-reminder}
 [a]^2=[a][-1]=0\quad(a\in\C^\times),
 \qquad [-1]^2=0.
\end{equation}
Indeed, the Milnor identity $\{a,a\}=\{a,-1\}$ and a square
root $a=b^2$ give
$\{a,-1\}=2\{b,-1\}=\{b,(-1)^2\}=0$.
These are statements about constant symbols. They do not identify
$\mathbf K_2^{\MW}$ with $\mathbf K_2^{\mathrm M}$ as sheaves
on $\Sm_{\C}$: the residue fields of smooth complex schemes
need not be algebraically closed.

\subsection{Normal framings and homotopy purity}

Two equations for the embedded curve produce an actual map from
its complement to the punctured plane. This is more information
than an abstract identification of an invariant.

\begin{proposition}\label{prop:CI-reminder}
For every closed embedding $i:\A^1_k\hookrightarrow\A^3_k$,
its defining ideal is generated by two polynomials. In particular,
over $\C$ one can choose $f,g\in\C[x,y,z]$ with
\[
 I_{i(\A^1)}=(f,g),\qquad
 (f,g):\linecomplement\longrightarrow\A^2_{\C}\setminus\{0\}.
\]
\end{proposition}

\begin{proof}
The embedding is regular of codimension two, and its normal bundle
is a rank-two vector bundle on $\A^1_k$, hence is free. The
complete-intersection theorem for codimension-two embeddings with
trivial normal bundle then applies. We use the formulation and
references collected in \cite[Theorem 2.6]{WW26}. Triviality of
the normal bundle alone is a local statement; the
complete-intersection theorem is the additional input that yields
global generators of the ideal.
\end{proof}

The ordered pair $(f,g)$ trivializes the conormal bundle
$I/I^2$, and dually the normal bundle. Homotopy purity identifies
the quotient of the ambient space by the complement with the
Thom space of that normal bundle \cite[Theorem 3.2.23]{MV99}.
Since the curve and the ambient affine space are both
$\A^1$-contractible, comparison with the zero section of
$\A^2$ gives
\begin{equation}\label{eq:suspended-normal-reminder}
 \Sigma_s(\linecomplement)
 \xrightarrow[\ \Sigma_s(f,g)\ ]{\sim}
 \Sigma_s(\A^2_{\C}\setminus\{0\}).
\end{equation}
This calculation already determines the reduced
$\A^1$-homology sheaves, and in particular
\begin{equation}\label{eq:H1-reminder}
 \mathbf H_1^{\A^1}(\linecomplement)
       \cong\mathbf K_2^{\MW}.
\end{equation}
It does not permit one to cancel the suspension in
\eqref{eq:suspended-normal-reminder} or to replace the Hurewicz
map by an isomorphism without further work. The proof of the
full fundamental-sheaf calculation addresses precisely this
remaining nonabelian problem.

\begin{remark}
An ordered normal framing fixes the orientation of all the
comparison maps. Replacing it by a matrix with constant determinant
$d\in k^\times$ changes the induced Milnor--Witt degree by
$\langle d\rangle$. Over $\C$ this factor equals $1$.
The corresponding assertion over another field must retain the
quadratic factor; we return to this in the appendix.
\end{remark}

\section{Nodal projections and height models}
\label{sec:geometry}

The normal bundle gives a comparison with the punctured affine plane,
but it initially controls only the suspended complement. To recover the
fundamental sheaf before suspension, we use the embedding and
choose a plane projection. The finitely many double points of that
projection will provide the finite list of possible defects. An affine
line torsor then lets us extend the height globally without changing the
torsors that classify maps to strongly invariant group sheaves.

\begin{lemma}\label{geom:connected}
The complement of a codimension-at-least-two closed subset of $\A^3_\C$
is $\A^1$-connected. In particular $\A^3_{\C}\setminus i(\A^1_{\C})$ is $\A^1$-connected.
\end{lemma}
\begin{proof}
Over any finitely generated extension $F/\C$, two $F$-points of the
complement can be joined through a sufficiently general intermediate
$F$-point by two affine lines which avoid the closed subset.
For a fixed endpoint, the union of lines meeting that subset has
dimension at most two; hence the forbidden intermediate points lie in
a proper closed subset of $\A^3_F$. Since $F$ is infinite, a permitted
intermediate point exists. This gives $\A^1$-chain connectedness over
all such fields, and the standard chain-connectedness criterion for
smooth schemes implies $\A^1$-connectedness; see
\cite[Section~2]{AD09}.
\end{proof}

\begin{proposition}\label{geom:normal-map}
The ideal of $i(\A^1)$ is generated by a regular sequence $F_1,F_2$.
The normal map $\phi:\bigl(\A^3_{\C}\setminus i(\A^1_{\C})\bigr)\to \bigl(\A^2_{\C}\setminus\{0\}\bigr)$ induces an equivalence
$\susp\phi:\susp \bigl(\A^3_{\C}\setminus i(\A^1_{\C})\bigr)\simeq\susp \bigl(\A^2_{\C}\setminus\{0\}\bigr)$, hence an isomorphism on reduced
$\A^1$-chains. In particular $\HA{1}(\bigl(\A^3_{\C}\setminus i(\A^1_{\C})\bigr))=\mathbf K_2^{\mathrm{MW}}$. Its map on fundamental
sheaves has a meridian splitting.
\end{proposition}
\begin{proof}
The complete-intersection assertion is \cite[Theorem~2.6]{WW26}.
It can also be obtained from the affine codimension-two Serre
construction: the smooth line is subcanonical, so its oriented ideal is
the zero locus of a section of a rank-two vector bundle on $\A^3$;
Quillen--Suslin makes that bundle free. The generators frame the normal
bundle along the line. Homotopy purity for the map of pairs
$(\A^3,i(\A^1))\to(\A^2,0)$ identifies its cofiber map with the map of
rank-two Thom spaces. It is an isomorphism on framed normal bundles and
an $\A^1$-equivalence on the bases $\A^1\to\Spec \C$, and hence is a Thom
equivalence. Both ambient affine spaces are contractible, so this is the
stated suspended equivalence. Reduced chains carry simplicial suspension
to a shift, proving the chain assertion. Morel's computation
$\piA{1}(\bigl(\A^2_{\C}\setminus\{0\}\bigr))=\HA{1}(\bigl(\A^2_{\C}\setminus\{0\}\bigr))=\mathbf K_2^{\mathrm{MW}}$ gives the homology assertion.

For the splitting, choose a transverse affine plane meeting the embedded
line in finitely many distinct reduced points $p_0,\ldots,p_s$.
Put $p_0=0$ and choose plane coordinates so that $p_j=(x_j,y_j)$ has
$x_jy_j\ne0$ for $j>0$. Choose $P\in \C[b]$ with $P(0)=1$ and
$P(y_j)=0$. The morphism
\[
 \mu:\bigl(\A^2_{\C}\setminus\{0\}\bigr)\to \bigl(\A^3_{\C}\setminus i(\A^1_{\C})\bigr),\qquad (a,b)\longmapsto(aP(b),b)
\]
through this plane avoids every other puncture: over $b=y_j$ its first
coordinate is zero, whereas $x_j\ne0$. Its unique transverse preimage of
the embedded curve is the removed origin. Applying the same purity
argument to this map of pairs proves that $\susp\mu$ is an equivalence.
Consequently $\phi\mu$ induces an isomorphism on $\HA{1}(\bigl(\A^2_{\C}\setminus\{0\}\bigr))$, and thus
on $\piA{1}(\bigl(\A^2_{\C}\setminus\{0\}\bigr))$ since the latter is already abelian and strictly
invariant. Normalize $\mu_*$ by the inverse of this automorphism. The
result is a splitting of $\phi_*$.
\end{proof}

\begin{proposition}\label{geom:projection}
A nonempty Zariski open set of linear directions gives finite birational
immersive projections $\A^1\to\Gamma\subset\A^2$ with only ordinary
nodes in the affine image. The two points over a node have distinct
heights in the remaining affine coordinate.
\end{proposition}
\begin{proof}
If the curve is planar, project transversely to its plane; the image is
smooth. Suppose it is nonplanar, and write its three polynomial
coordinates as $\gamma(t)$. Every nonconstant projected coordinate $f(t)$
makes $\C[t]$ finite over $\C[f]$: the polynomial equation $f(t)-f=0$ can
be divided by its nonzero leading coefficient to be monic in $t$.
The same finite generators work over the intermediate ring $\C[f,g]$,
so every nonconstant such projection is finite onto its closed image.

The tangent directions $[\gamma'(t)]$ form a subset of dimension at most
one in the direction plane $\mathbb P^2$. The derivative never vanishes.
Avoiding its closure makes the projection immersive. Put
\[
 E(s,t)=\det(\gamma(s)-\gamma(t),\gamma'(t),\gamma'(s)).
\]
It is not identically zero. Indeed its coefficient of $h^4$ for $s=t+h$
is $-\det(\gamma',\gamma'',\gamma''')/12$. This last determinant is the
Wronskian of the coordinate derivatives. In characteristic zero its
identical vanishing would make those derivatives linearly dependent
over $\C$, yielding an affine plane containing the curve.
Thus the locus $E=0$ off the diagonal has dimension at most one.
Avoid its chord directions. Projected double identifications then have
distinct tangent lines, because tangent dependence in the quotient is
exactly the vanishing of $E$.

The chord-direction map $(s,t)\mapsto[\gamma(s)-\gamma(t)]$ has rank two
where $E\ne0$. It is dominant and generically finite. After shrinking
the direction plane, there are only finitely many identified pairs;
the projection is generically injective, hence birational in
characteristic zero. Finally, the classical trisecant lemma, valid for
the integral possibly singular projective closure, bounds its family of
trisecant lines by dimension one \cite[Theorem~2]{KKT08}. The directions
of trisecants not contained at infinity also have dimension at most one.
Avoiding them excludes every triple affine fiber. A finite immersive
normalization with no triple fibers and transverse double fibers has
only ordinary nodes. Completing the projection coordinates to coordinates
on $\A^3$ shows that equal heights at a double fiber would violate the
injectivity of the original embedding.
\end{proof}

Fix this projection, write $i(t)=(f(t),g(t),h(t))$, and let $P$ be its
node set. An affine coordinate $x$ separates the nodes. Polynomial
interpolation followed by the shear $y\mapsto y-p(x)$ puts them at
$(a_1,0),\ldots,(a_r,0)$. Put $q(x)=\prod_i(x-a_i)$ and
\[
 B=\A^2-P,\quad D=\A^2-\Gamma,\quad
 \Gamma^\circ=\Gamma-P,\quad
 S=\{q(x)\alpha+y\beta=1\}\longrightarrow B.
\]
For $r>0$, $S$ is smooth affine and an affine-line torsor: the row
$(q,y)$ is unimodular on $B$, and its solutions form a torsor under its
rank-one kernel. For $r=0$ take $S=B=\A^2$.

\begin{lemma}\label{geom:linear-height}
One can choose a global regular height extension $\widehat h$ on $S$
which is linear in $\alpha,\beta$. Moreover, in a split henselian node
chart $R=\C\{u,v\}$ with $\Gamma$ defined by $uv=0$, the two canonical
sections of $S$ over $D(u)$ and $D(v)$ evaluate $\widehat h$ to regular
functions of $R$, whose closed values are the two branch heights.
\end{lemma}
\begin{proof}
For an affine nodal curve with normalization $\A^1$, its coordinate ring
is the subring of $\C[t]$ consisting of polynomials with equal values at
the two preimages of each node. This follows locally from
$\C\{u,v\}/(uv)\subset \C\{u\}\oplus \C\{v\}$, whose image is the pairs
with the same closed value, and then globally by the finite
normalization. Both $q(f(t))h(t)$ and $g(t)h(t)$ vanish at every node
preimage and therefore descend to $\Gamma$. Lift them to polynomials
$A(x,y),B_1(x,y)$ on $\A^2$ and set
\[
 \widehat h=A\alpha+B_1\beta.
\]
Its restriction to the pullback of $\Gamma^\circ$ is
$h(q\alpha+y\beta)=h$, as required. If there are no nodes, $h$ itself
descends to the smooth curve $\Gamma$ and can be lifted directly to
$\A^2$.

At a node, $(q,y)$ and $(u,v)$ are regular systems of parameters, so
$(q,y)=(u,v)M$ for $M\in\mathrm{GL}_2(R)$. Set
$(a,b)^t=M(\alpha,\beta)^t$ and
$(A',B')=(A,B_1)M^{-1}$. Then
$S_R=\{ua+vb=1\}$ and $\widehat h=A'a+B'b$.
On $\Gamma$, $(A',B')=h(u,v)$. Consequently $A'$ vanishes identically
on $u=0$ and $B'$ on $v=0$. Equality first on the punctured branches
extends to the whole branches by integrality. Write
$A'=uH_u$, $B'=vH_v$ with $H_u,H_v\in R$. The canonical sections
$(a,b)=(1/u,0)$ and $(0,1/v)$ have heights $H_u$ and $H_v$.
On $v=0$ the first height is the branch height, and on $u=0$
the second is the branch height. Their closed values are distinct
by Proposition~\ref{geom:projection}.
\end{proof}

\begin{proposition}\label{geom:height-model}
Let $\bigl(\A^3_{\C}\setminus i(\A^1_{\C})\bigr)^\circ$ be obtained from $\bigl(\A^3_{\C}\setminus i(\A^1_{\C})\bigr)$ by removing the vertical axes over $P$.
There is a pointed equivalence on ${\mathcal Q}$-torsor groupoids, for every strongly
invariant ${\mathcal Q}$, between $\bigl(\A^3_{\C}\setminus i(\A^1_{\C})\bigr)^\circ$ and
\[
 W=(B\times\A^1_w)-(\Gamma^\circ\times0).
\]
Restriction from $\bigl(\A^3_{\C}\setminus i(\A^1_{\C})\bigr)$ to $\bigl(\A^3_{\C}\setminus i(\A^1_{\C})\bigr)^\circ$ is fully faithful on these groupoids.
\end{proposition}
\begin{proof}
On the common affine-line-torsor cover over $B$, the change
$w=z-\widehat h$ identifies the pullbacks of the two complements: over
$\Gamma^\circ$ precisely the graph of the height is removed. Both
projections of this common model are affine-line torsors, hence
$\A^1$-equivalences. Mapping into the $\A^1$-local space $B{\mathcal Q}$ gives the
torsor-groupoid equivalence. Choose a rational point over $d_*\in D$ to
transport pointings.

The omitted subset has codimension two in the smooth $\bigl(\A^3_{\C}\setminus i(\A^1_{\C})\bigr)$. Given two
torsors on $\bigl(\A^3_{\C}\setminus i(\A^1_{\C})\bigr)$ and an isomorphism over $\bigl(\A^3_{\C}\setminus i(\A^1_{\C})\bigr)^\circ$, take a common
Nisnevich cover trivializing them. On a member the isomorphism is a
${\mathcal Q}$-section off codimension two. Unramifiedness extends it uniquely over
that member, and uniqueness glues the extensions. This also proves
injectivity on automorphisms.
\end{proof}

\begin{lemma}\label{geom:filling}
If $r>0$, there is a nonempty open $Z\subset\A^1_z$ and a morphism
$j:\A^2\times Z\to \bigl(\A^3_{\C}\setminus i(\A^1_{\C})\bigr)$ such that its restriction to $B\times Z$ maps to
$\bigl(\A^3_{\C}\setminus i(\A^1_{\C})\bigr)^\circ$. Every ${\mathcal Q}$-torsor on the complement pulls back trivially along $j$.
\end{lemma}
\begin{proof}
Both $f,g$ are nonconstant, for otherwise the irreducible projected image
would be a smooth line. Since $h(t)$ is integral over $\C[g(t)]$, the
equation $R(y,z)$ of the $(g,h)$-image can be made monic in $z$.
The polynomial $R_0(z)=R(0,z)$ is nonzero: irreducibility would otherwise
force $R$ to be a scalar multiple of $y$, contrary to nonconstancy of $g$.
Remove its zeros and every height $h(t)$ with $f(t)=0$, and put
$\Lambda(y,z)=R(y,z)/R_0(z)$ on the remaining $Z$. Define
\[
 j(a,b,z)=(a\Lambda(b,z),b,z).
\]
An intersection with the curve would force $R(b,z)=0$, hence $f(t)=0$
and an excluded height. Also $\Lambda(0,z)=1$, so the inverse image of the
vertical axes is exactly $P\times Z$.

A Nisnevich torsor on $Z$ is generically trivial and is trivial on some
nonempty Zariski open $E\subset Z$. Glue it along $E$ to the trivial
torsor on $\A^1-(Z-E)$. The result is a torsor on $\A^1$, which is
trivial by strong invariance and $H^1_\Nis(\C,{\mathcal Q})=*$. Thus
$H^1_\Nis(Z,{\mathcal Q})=*$, and $\A^1$-invariance gives
$H^1_\Nis(\A^2\times Z,{\mathcal Q})=*$. This proves the asserted triviality.
\end{proof}

\begin{lemma}\label{geom:node-charts}
Let $c$ be a height clutching section for a $W$-torsor with the left
convention $s_D=s_Hc$. On the punctured henselian normal plane over the
generic missing node axis, the induced $\bigl(\A^3_{\C}\setminus i(\A^1_{\C})\bigr)^\circ$-torsor has chart
transition
\[
 c(z-H_u)c(z-H_v)^{-1},
\]
where $H_u,H_v\in R$ are the regular functions of
Lemma~\ref{geom:linear-height}. Both factors use the same global section
$c$.
\end{lemma}
\begin{proof}
Use the linear height extension fixed above. The canonical sections of
$S_R$ over $D(u)$ and $D(v)$ identify the complement charts with
the height model by $w=z-H_u$ and $w=z-H_v$. The $D\times\A^1$ chart
trivialization of the $W$-torsor pulls back to the common $S$-model and
descends to a single trivialization over $D\times\A^1_z$: pullback along
its affine-line torsor is fully faithful on ${\mathcal Q}$-torsors. Thus the two
canonical local sections refer to the same $D$-trivialization. Comparing
each to its pulled-back height-open trivialization gives
$s_D=s_uc(z-H_u)=s_vc(z-H_v)$. We fix the ordered transition by
$s_v=s_ug$; consequently $g=c(z-H_u)c(z-H_v)^{-1}$, with left $u$-chart
and right $v$-chart gauges. At the generic axis, $z-H_u$
and $z-H_v$ are units in the henselian normal rings, since their closed
values are the nonzero rational functions $z-h_2,z-h_1$. This makes both
chart trivializations available there.
\end{proof}

\section{Contractions and repeated torus coordinates}\label{sec:contractions}

We develop the group identities arising from repeated and independent
pointed torus coordinates.

\begin{lemma}\label{con:square}
Let ${\mathcal Q}$ be strongly $\A^1$-invariant. There is a natural bijection
\[
 \Map_*(\Gm^{\wedge2},{\mathcal Q})=\Hom(\mathbf K_2^{\mathrm{MW}},{\mathcal Q}).
\]
If a ${\mathcal Q}$-valued family is pointed in two independent torus coordinates,
restricting those coordinates to the diagonal replaces $[w][w]$ by
$[w][-1]$. The assertion holds with additional independent parameters.
\end{lemma}
\begin{proof}
The standard equivalence $\bigl(\A^2_{\C}\setminus\{0\}\bigr)\simeq\susp\Gm^{\wedge2}$ and the adjunction
between simplicial suspension and looping identify the left hand side
with pointed maps $\bigl(\A^2_{\C}\setminus\{0\}\bigr)\to B{\mathcal Q}$. The space $B{\mathcal Q}$ is an $\A^1$-local connected
$1$-type, so these maps are homomorphisms from $\piA{1}(\bigl(\A^2_{\C}\setminus\{0\}\bigr))=\mathbf K_2^{\mathrm{MW}}$ to ${\mathcal Q}$.
Under this identification the torus family evaluates the universal symbol
$[x][y]$. The Milnor--Witt identity $[w][w]=[w][-1]$ gives the diagonal
assertion. With parameters, work at the generic fields of their smooth
components and use continuous base change and unramified injectivity to
extend the equality. Alternatively, retaining a torus parameter replaces
${\mathcal Q}$ by the strong group $X\mapsto {\mathcal Q}(X\times\Gm)$: its classifying space
is the internally connected $\A^1$-local $1$-type $\Map(\Gm,B{\mathcal Q})$ by
Morel's $\Gm$-loop theorem \cite[Theorem~3.1.19]{AFH22}.
\end{proof}

\begin{lemma}\label{con:constants}
For every $c\in \C^*$,
\[
 2[-1]=0,\qquad [-1][-1]=0,\qquad [c]^2=[c][-1]=0.
\]
These equalities remain true after any extension of $\C$. Positive
Milnor--Witt groups of $\C$ are the corresponding Milnor groups and are
$2$-divisible.
\end{lemma}
\begin{proof}
The fundamental ideal of the Witt ring of the algebraically closed field
$\C$ is zero. The comparison $K_n^{\mathrm{MW}}(\C)\to K_n^M(\C)$ is therefore
an isomorphism for $n\ge1$. The class $\{-1\}$ has order dividing two.
Write $c=d^2$ to get
$\{c,-1\}=2\{d,-1\}=\{d,(-1)^2\}=0$.
The relation $\{c,c\}=\{c,-1\}$ proves the square identity; take $c=-1$
for the second assertion. Functoriality proves persistence under field
extension. A pure positive symbol is twice the symbol obtained by choosing
a square root in its first coordinate; finite sums give $2$-divisibility.
\end{proof}

\begin{proposition}\label{con:structure}
For every strongly invariant ${\mathcal Q}$ over $\C$, the sheaf ${\mathcal Q_{-1}}$ has
nilpotence class at most two and all its commutators have exponent dividing
two. The sheaves ${\mathcal Q_{-n}}$ for $n\ge2$ are strictly $\A^1$-invariant abelian
sheaves. In particular every square in ${\mathcal Q_{-1}}$ lies in its center.
\end{proposition}
\begin{proof}
For sections $a,b\in {\mathcal Q_{-1}}$, give their outer torus coordinates independent
names and form $[a(w_1),b(w_2)]$. It is fully pointed. On the diagonal,
Lemma~\ref{con:square} evaluates a homomorphism from $\mathbf K_2^{\mathrm{MW}}$ on $[w][-1]$.
Twice this symbol is zero, so $[a,b]^2=1$.

For a triple commutator use the fully pointed family
$[a(w_1),[b(w_2),c(w_3)]]$. Restricting all three coordinates to $w$
first replaces $(w_1,w_2)$ by $(w,-1)$, then replaces $(w_1,w_3)$ by
$(w,-1)$. The result has two independent slots set to $-1$ and is trivial
by Lemmas~\ref{con:square}--\ref{con:constants}. All triple commutators
in ${\mathcal Q_{-1}}$ vanish.

A commutator of two fully pointed two-torus families is first formed on
two independent pairs. Identifying the pairs produces two repetitions,
and hence two $-1$ slots. It is trivial. Thus ${\mathcal Q_{-2}}$ is abelian, and so
are its further contractions. Contraction preserves strong invariance
\cite[Proposition~3.1.18]{AFH22}. An abelian strongly invariant sheaf over
a perfect field is strictly invariant \cite[Theorem~3.1.8]{AFH22}.

The arguments at generic fields extend to all smooth bases by
unramifiedness. Finally, in a group of class at most two,
$[x^2,y]=[x,y]^2=1$, proving the centrality of squares.
\end{proof}

\begin{proposition}\label{con:bracket}
For $a,b\ge2$, the commutator of families on disjoint torus blocks defines
a natural biadditive morphism
\[
 \beta_{a,b}:{\mathcal Q_{-a}}\times {\mathcal Q_{-b}}\longrightarrow {\mathcal Q_{-(a+b)}}.
\]
It commutes with evaluation of additional pointed torus slots. With any
fixed order of the coordinates, for $c\in \C^*$ one has, up to the fixed
coordinate-permutation sign,
\[
 \beta_{a,b}([c]x,y)=[c]\beta_{a+1,b}(x,y),
 \qquad
 \beta_{a,b}(x,[c]y)=[c]\beta_{a,b+1}(x,y).
\]
The same assertion holds for products of constant symbols.
\end{proposition}
\begin{proof}
The commutator is fully pointed in the $a+b$ independent coordinates, so
it is a section of the indicated contraction. To check the first
additivity assertion, use the group identity for $[xx',y]$.
The extra conjugation term is a triple commutator in which $x$ and $x'$
use the same $a$-tuple. Give these two occurrences independent $a$-tuples
before taking the diagonal. Applying Lemma~\ref{con:square} in two of the
coordinate pairs, possible because $a\ge2$, produces two independent
$-1$ slots. Their symbol is zero and the conjugation term vanishes.
The resulting bracket values commute because ${\mathcal Q_{-(a+b)}}$ is abelian.
The second additivity assertion uses two repeated coordinates in the
$b$-block in exactly the same manner.

For the evaluation assertion, introduce the additional slot before
forming the commutator. Evaluating it in the input or in the output gives
the same ${\mathcal Q}$-valued family. Reordering slots gives only the
specified sign in the abelian contraction tower. Iteration handles
several slots.
\end{proof}

\begin{lemma}\label{con:mixed}
Let $p,q\in {\mathcal Q_{-2}}(\C)$ and $\ell,m\in {\mathcal Q_{-3}}(\C)$.
As sections of ${\mathcal Q_{-1}}(\Gm^{(u)}\times\Gm^{(v)})$, the fully pointed mixed factor
$a_\ell(u,v)$ commutes with $e_u(p)$, $e_v(q)$, and $a_m(u,v)$.
The notation regards the common outer coordinate of ${\mathcal Q_{-1}}$ as
retained.
\end{lemma}
\begin{proof}
For $[a_\ell,e_u(p)]$, first use independent copies of their common
outer coordinate and their common $u$ coordinate. On the diagonal there
are two repeated slots. Two applications of Lemma~\ref{con:square} make
the commutator trivial. For $e_v(q)$ use the outer and $v$ coordinates.
For the two mixed factors use, for example, their common $u$ and $v$
coordinates. These are equalities of sections.
\end{proof}

\begin{remark}\label{con:actions}
For $n\ge2$, evaluating $m$ additional pointed slots of ${\mathcal Q_{-(n+m)}}$
at units $c_1,\ldots,c_m$ is the canonical action
$[c_1]\cdots[c_m]:{\mathcal Q_{-(n+m)}}\to {\mathcal Q_{-n}}$ on strict abelian contractions.
The pointed torus map into this abelian target factors through the
corresponding Milnor--Witt symbol, so a zero symbol product acts as zero.
This action is on the strict abelian contraction tower.
\end{remark}

\section{Full framed specialization}
\label{sp:section}

The node calculation requires a group homomorphism specializing
transition functions and chart gauges, including their constant and
pointed torus parts. Its construction has two steps: extension across
the divisor and computation of the closed value.

Throughout this section, $\mathcal Q$ is a strongly
$\A^1$-invariant Nisnevich sheaf of groups on $\Sm_{\C}$.
Sections on essentially smooth rings and henselizations mean their
continuous extensions from finite smooth stages. For
$r\in\mathcal Q_{-1}(S)$ and an invertible function $u$, let
$e_u(r)$ denote evaluation of the pointed torus family $r$ at $u$.
Evaluation is a group homomorphism in $r$. No multiplicativity is
asserted for a general nonabelian residue.

We first make the finite-stage geometry explicit. Let $S$ be smooth
affine. An affine \'etale neighborhood used for the henselization
along $u=0$ has the form
\begin{equation}\label{sp:finite-pair}
 \begin{array}{ccc}
 S&\lhook\joinrel\longrightarrow&X\\
 \big\Vert&&\big\downarrow\scriptstyle\theta\\
 S&\lhook\joinrel\longrightarrow&S\times\A^1_u.
 \end{array}
\end{equation}
The square is cartesian, and its bottom arrow is the zero section.
We may arrange that $\theta^{-1}(u=0)$ is exactly the displayed
copy of $S$. Indeed the prescribed section of the \'etale zero
fiber is open and closed; after shrinking to a principal neighborhood
of that section, the other components are removed. We retain the
projection $X\to S$ throughout. Write $X_u=X\setminus S$.

Since the ideal of $S$ in $X$ is generated by $u$, its deformation
to the normal cone is the following explicit smooth affine scheme:
\begin{equation}\label{sp:dnc-chart}
 \mathcal E=D_SX
 =\Spec\bigl(\mathcal O(X)[\lambda,z]/(u-\lambda z)\bigr)
 =X\mathbin{\times}_{S\times\A^1_u}
      (S\times\A^1_\lambda\times\A^1_z).
\end{equation}
The last map to $S\times\A^1_u$ sends $u$ to $\lambda z$.
Thus $\mathcal E$ is \'etale over $S\times\A^2_{\lambda,z}$.
Its divisor $z=0$ is exactly $S\times\A^1_\lambda$, with
normal coordinate $z$. The two fibers and the map
$p:\mathcal E\to X$ fit into the commutative diagram
\begin{equation}\label{sp:dnc-fibers}
 \begin{array}{ccccc}
 S\times\A^1_z&\xrightarrow{\ \lambda=0\ }&\mathcal E
      &\xleftarrow{\ \lambda=1\ }&X\\
 \big\downarrow\scriptstyle\mathrm{pr}_S
      &&\big\downarrow\scriptstyle p
      &&\big\downarrow\scriptstyle\id\\
 S&\lhook\joinrel\longrightarrow&X&\xleftarrow{\ \id\ }&X.
 \end{array}
\end{equation}
The first fiber is the framed normal bundle, and the second has
$z=u$. All applications of purity below take place on these finite
smooth diagrams.

\begin{lemma}
\label{sp:supported}
Let $Z\hookrightarrow X$ be a smooth divisor in a smooth scheme,
with a specified normal framing. There is a natural equivalence of
discrete spaces
\begin{equation}\label{sp:supported-eq}
 \Map_*(X/(X\setminus Z),B\mathcal Q)
       \simeq\mathcal Q_{-1}(Z).
\end{equation}
If a strongly invariant group $\mathcal L$ acts on $\mathcal Q$
by sheaf automorphisms, the action of a section
$a\in\mathcal L(X)$ on this relative mapping space corresponds
to the action of $a|_Z$ on the right-hand side.
\end{lemma}

\begin{proof}
Framed homotopy purity identifies the quotient with
$\mathbb P^1\wedge Z_+$
\cite[Section~3, Theorem~2.23]{MV99}. The equivalence
$\mathbb P^1\simeq S_s^1\wedge(\Gm,1)$ permits the
simplicial suspension to be moved across the pointed mapping object.
Its target then becomes $\Omega_sB\mathcal Q=\mathcal Q$,
which is discrete. The remaining pointed maps are precisely
$\mathcal Q_{-1}(Z)$. In particular, looping precedes taking sections over $Z$,
even when $Z$ is not local.

A section of $\mathcal L(X)$ acts on a
$\mathcal Q$-torsor on $X$ by change of structure group through
the resulting automorphism of $\mathcal Q$ over $X$. It acts
also on a chosen trivialization off $Z$. Every such automorphism
fixes the identity, so this is an action on the relative pointed
mapping space. It is compatible with pullback of pairs.

Pull $a$ back to $D_ZX$. The two inclusions of the original and
normal fibers give equivalences on the quotients by their framed
zero divisors. In the finite model above these maps are
\[
 X/X_u\longrightarrow
 \mathcal E/(\mathcal E\setminus(S\times\A^1_\lambda))
 \longleftarrow
 (S\times\A^1_z)/(S\times\Gm^{(z)}).
\]
Purity identifies their normal bundle maps with the identity
on the normal coordinate and with the two inclusions
$S\to S\times\A^1_\lambda$. Thus they are equivalences.
The pullback actions commute with these maps on the relative
torsor groupoids. On the normal fiber, diagram
\eqref{sp:dnc-fibers} makes $p^*a$ the pullback of $a|_Z$.
Transporting this action through the two equivalences proves
the assertion.
\end{proof}

\begin{proposition}\label{sp:rho}
Let $S$ be smooth affine and put
\[
 A=\mathcal O(S)\{u\},\qquad
 J_u=\ker\bigl(\mathcal Q(A)\longrightarrow\mathcal Q(S)\bigr),
\]
where braces denote henselization along the whole divisor $u=0$.
Then $J_u$ is normal in $\mathcal Q(A[1/u])$, and coefficient
pullback and evaluation induce an isomorphism
\begin{equation}\label{sp:quotient}
 \mathcal Q(S\times\Gm^{(\xi)})\xrightarrow{\sim}
       \mathcal Q(A[1/u])/J_u.
\end{equation}
Its inverse followed by the quotient map is a homomorphism
\[
 \rho_u^{\mathcal Q}:\mathcal Q(A[1/u])
       \longrightarrow\mathcal Q(S\times\Gm^{(\xi)}),
\]
characterized by
\begin{equation}\label{sp:rho-rule}
 \rho_u^{\mathcal Q}\bigl(q_{\mathrm{reg}}e_u(r)\bigr)
  =(q_{\mathrm{reg}}|_S)e_\xi(r),\qquad
  r\in\mathcal Q_{-1}(S),\quad q_{\mathrm{reg}}\in\mathcal Q(A).
\end{equation}
It is natural under coefficient-base change and homomorphisms of
strongly invariant group sheaves. The same construction is valid
on the continuous coefficient limits used below.
\end{proposition}

\begin{proof}
We separate factorization, regularity of conjugates, and their
closed values.

\emph{Factorization at a finite stage.}
A section of $\mathcal Q(A[1/u])$ is represented on $X_u$ for
a finite pair \eqref{sp:finite-pair}. The mapping-space fiber
sequence at the trivial torsor is
\[
 \Map_*(X/X_u,B\mathcal Q)\longrightarrow
 \Map(X,B\mathcal Q)\longrightarrow\Map(X_u,B\mathcal Q).
\]
Its fiber is the discrete space $\mathcal Q_{-1}(S)$ by
Lemma~\ref{sp:supported}. The automorphism groups of the two
displayed trivial torsors are $\mathcal Q(X)$ and
$\mathcal Q(X_u)$. Every fiber component is the boundary of
an evaluation section $e_u(r)$ on $X_u$. To check this,
first use the product pair $S\times\A^1_u$, where evaluation
at $u=1$ splits a torus section into its constant and pointed
parts. Pull back along \eqref{sp:finite-pair}; framed purity
identifies the induced map on the relative fibers with the
identity of $\mathcal Q_{-1}(S)$.

Consequently the action of $\mathcal Q(X_u)$ on the fiber is
transitive, with stabilizer $\mathcal Q(X)$. With the left
clutching convention, its exact homotopy sequence gives
\begin{equation}\label{sp:factorization}
 \mathcal Q(X)\backslash\mathcal Q(X_u)
       \cong\mathcal Q_{-1}(S),\qquad
 q=q_{\mathrm{reg}}e_u(r).
\end{equation}
Passing to the filtered section limit gives the same factorization
over $A$. Neither the fiber sequence nor the transitivity argument
asserts that all torsors on $X$ or $A$ are trivial.

\emph{Regularity of the conjugate.}
Take $j\in J_u$ and $r\in\mathcal Q_{-1}(S)$, and put
$e=e_u(r)$. The subgroup $J_u$ is already normal in
$\mathcal Q(A)$. By Lemma~\ref{sp:supported}, conjugation by
$j$ acts trivially on the supported class of $e$, since $j|_S=1$.
Equality of the corresponding left cosets in
\eqref{sp:factorization} implies that $eje^{-1}$ extends across
$u=0$. Thus the commutator $[e,j]$ is regular; its closed value
remains to be computed.

\emph{Its closed value.}
Represent $j$ on a finite $X$ and use the smooth scheme
$\mathcal E$ in \eqref{sp:dnc-chart}. Let $j_{\mathcal E}=p^*j$.
On $D(z)$ consider the section
\[
 \kappa=[e_z(r),j_{\mathcal E}].
\]
The restriction of $j_{\mathcal E}$ to the whole divisor
$z=0$ is $1$. Apply the preceding finite-stage supported-class
argument with coefficient base $S\times\A^1_\lambda$ and
normal coordinate $z$. It makes $\kappa$ regular at every
generic point of this divisor. It is already regular on $D(z)$,
so unramifiedness extends it uniquely over $\mathcal E$.
On the normal fiber $\lambda=0$, the map $p$ factors through
$S$; hence $j_{\mathcal E}=1$ there. The extended commutator
is therefore $1$ on that fiber, including its point $z=0$.
On the other hand its restriction to the entire zero divisor is
an element of
\[
 \mathcal Q(S\times\A^1_\lambda)=\mathcal Q(S).
\]
This restriction is independent of $\lambda$. It has value $1$
at $\lambda=0$, so it has value $1$ at $\lambda=1$, where
$z=u$. We have proved $[e_u(r),j]\in J_u$. Applying the same
argument to $r^{-1}$ gives invariance under conjugation in both
directions. Factorization \eqref{sp:factorization} now proves
normality under all of $\mathcal Q(A[1/u])$.

The map in \eqref{sp:quotient} is induced by $\xi\mapsto u$.
It is surjective by factorization, since a regular section and
its constant closed value differ by an element of $J_u$.
If $q_0e_u(r)$, with $q_0\in\mathcal Q(S)$, belongs to $J_u$,
its supported class first gives $r=1$; its closed value then
gives $q_0=1$. This proves injectivity and
\eqref{sp:rho-rule}. Pullback of coefficients, or a homomorphism
of target group sheaves, preserves regular sections, their closed
values, and evaluation. The characterizing formula proves both
naturalities. All sections and equalities used above occur at
finite stages and pass to the indicated continuous limits.
\end{proof}

\begin{lemma}
\label{sp:moving}
In the notation of Proposition~\ref{sp:rho}, let $a(u)\in A^*$
have closed value $a_0\in\mathcal O(S)^*$. For
$c\in\mathcal Q_{-1}(A[1/u])$ one has
\begin{equation}\label{sp:moving-eq}
 \rho_u^{\mathcal Q}\bigl(c(a(u))\bigr)
       =\bigl(\rho_u^{\mathcal Q_{-1}}c\bigr)(a_0).
\end{equation}
For $H\in\mathcal O(S)$ and
$q\in\mathcal Q(S\times\Gm^{(w)})$,
\begin{equation}\label{sp:infinity-change}
 \rho_u^{\mathcal Q}\bigl(q(s,u^{-1}-H(s))\bigr)
       =q(s,\xi^{-1}).
\end{equation}
Both identities retain the entire $\xi$-family.
\end{lemma}

\begin{proof}
Factor $c=c_{\mathrm{reg}}e_u(r)$ in the strong group
$\mathcal Q_{-1}$, using Proposition~\ref{sp:rho}. Evaluating
the regular factor at the unit $a(u)$ produces a regular
$\mathcal Q$-section, with closed value obtained by evaluating
at $a_0$. It remains to compare the evaluations of the
two-torus family representing $r$.

Write that family as $f(s,w,\xi)$. Choose a finite pair
$(X,S)$ on which $a(u)$ is a unit, retain its projection to
$S$, and form \eqref{sp:dnc-chart}. Put
$a_{\mathcal E}=p^*a$. Consider the two maps of smooth pairs
\[
 \begin{array}{ccc}
 (\mathcal E,S\times\A^1_\lambda)
 &\xrightarrow{\ \gamma_a,\ \gamma_0\ }&
 (S\times\Gm^{(w)}\times\A^1_z,
                S\times\Gm^{(w)}\times\{0\}),
 \end{array}
\]
where
\[
 \gamma_a(s,\lambda,z)=(s,a_{\mathcal E},z),\qquad
 \gamma_0(s,\lambda,z)=(s,a_0,z).
\]
On the zero divisor $z=0$, both maps have $w=a_0$; both send
the framed normal coordinate to $z$. Naturality of purity thus
identifies their maps on the supported quotients. Pulling back
the section $f$ gives the comparison
\[
 \kappa=f(s,a_{\mathcal E},z)f(s,a_0,z)^{-1}
                   \quad\hbox{on }D(z).
\]
The equality of supported classes and
\eqref{sp:factorization}, applied with base
$S\times\A^1_\lambda$, make $\kappa$ regular across
the generic points of $z=0$. Unramifiedness extends it over
$\mathcal E$. On $\lambda=0$, one has
$a_{\mathcal E}=a_0$, so its value is $1$. Its restriction to
$z=0$ is constant in $\lambda$ by $\A^1$-invariance, and
therefore also has value $1$ at $\lambda=1$. Consequently
\[
 f(s,a(u),u)f(s,a_0,u)^{-1}\in J_u.
\]
This proves \eqref{sp:moving-eq} for the pointed factor; the
regular factor and multiplicativity finish the proof. The unit
$a_{\mathcal E}$ is the pullback of a unit on $X$.

For the second identity, begin with $q=e_w(r)$ for
$r\in\mathcal Q_{-1}(S)$. The two normal parameters
$u$ and $v=u/(1-uH)$ have the same zero section and the same
framed differential there. The coordinate comparison can be
made on the explicit smooth open
\[
 \mathcal E_H=D(1-\lambda zH)
       \subset S\times\A^1_z\times\A^1_\lambda.
\]
It contains the whole divisor $z=0$. The maps with normal
coordinates $z$ and $z/(1-\lambda zH)$ agree on this divisor
and on its normal framing. The quotient on $D(z)$ is
\[
 r(z^{-1}-\lambda H)r(z^{-1})^{-1}.
\]
Purity gives equality of its supported classes; hence this
quotient extends across the generic points of $z=0$, and
unramifiedness extends it over $\mathcal E_H$. It is $1$
on $\lambda=0$. Its restriction to the whole zero divisor
$S\times\A^1_\lambda$ is constant, so its closed value is
also $1$ at $\lambda=1$. This proves the identity in the full quotient. Finally
write a general $q$ as $q(1)e_w(r)$ and use the group
homomorphism $\rho_u^{\mathcal Q}$.
\end{proof}

\begin{lemma}\label{sp:mapping-group}
For a strongly invariant group $M$, the sheaf
$E_M(T)=M(T\times\Gm)$ is strongly invariant, and
\[
 \Map(\Gm,BM)\simeq BE_M,\qquad E_M=M_{-1}\rtimes M.
\]
\end{lemma}

\begin{proof}
Morel's loop theorem makes the pointed internal mapping object
$\Omega_{\Gm}BM$ connected, with fundamental sheaf $M_{-1}$
and no higher homotopy sheaves
\cite[Theorem~6.13]{Morel12}; see also
\cite[Theorem~3.1.19]{AFH22}. It is therefore $B(M_{-1})$.
Evaluation at $1$ gives a split fiber sequence from this object
to $\Map(\Gm,BM)$ and then to $BM$. The base and fiber are
internally connected; the middle object is a local $1$-type
whose loop sheaf is $M(\Gm\times-)$, with the stated
semidirect product. This proves the assertions. For a nonlocal
test scheme, the section groupoid of this internally connected
object can have nontrivial components. The argument discards
none of them.
\end{proof}

\begin{proposition}
\label{sp:double}
Let $R=\C\{u,v\}$, $S=\C\{u\}$, and $F=\Frac S$.
For every strongly invariant $M$ there is a group retraction
\[
 \tau_{(u,v)}^M:M(R[1/(uv)])
       \longrightarrow M(\Gm^{(u)}\times\Gm^{(v)})
\]
of coordinate evaluation, satisfying
\begin{align}\label{sp:chart-images}
 \tau^M(M(R[1/u]))&\subset M(\Gm^{(u)}),&
 \tau^M(M(R[1/v]))&\subset M(\Gm^{(v)}).
\end{align}
For $c\in\mathcal Q_{-1}(R[1/(uv)])$ and $H\in R$
with closed value $h$, the moving-height identity is
\begin{equation}\label{sp:height-freezing}
 \tau_{(u,v)}^{\mathcal Q}(c(z-H))
       =(\tau_{(u,v)}^{\mathcal Q_{-1}}c)(z-h).
\end{equation}
In this identity the left specialization is formed after the
canonical henselian coefficient change to the generic height
field $\C(z)$, and the right-hand side is then restricted to
that field. Thus both sides belong to
$\mathcal Q((\Gm^{(u)}\times\Gm^{(v)})_{\C(z)})$.
The same identity holds on smooth height
opens with the corresponding henselian neighborhoods of the
whole zero section.
\end{proposition}

\begin{proof}
The ring $S\{v\}$ is henselian as a local ring at $(u,v)$:
the pair with ideal $(v)$ is henselian, and its quotient $S$
is henselian local. The universal property of the local
henselization therefore gives a canonical map
$R\to S\{v\}$. The coefficient map $S\to F$ gives in turn
$S\{v\}\to F\{v\}$. Only these maps are needed; we do not
identify iterated henselizations with one another or assert
that ordinary localization preserves henselianity.

Since $S$ is a henselian discrete valuation ring with parameter
$u$, its fraction field is $S[1/u]$. Define the specialization
in the fixed order
\begin{equation}\label{sp:ordered-composition}
 \begin{split}
 M(R[1/(uv)])&\longrightarrow M(F\{v\}[1/v])
 \xrightarrow{\rho_v^M}E_M(F)\\
 &\hspace{25mm}\xrightarrow{\rho_u^{E_M}}
 E_M(\Gm^{(u)})
       =M(\Gm^{(u)}\times\Gm^{(v)}).
 \end{split}
\end{equation}
Lemma~\ref{sp:mapping-group} permits the second use of
Proposition~\ref{sp:rho}. Each full specialization retains a
torus family already expressed in its named coordinate. This
proves the retraction property.

If a section comes from $M(R[1/u])$, its image in $M(F\{v\})$
is $v$-regular. The first specialization evaluates at $v=0$,
so it has no retained $v$-dependence. Naturality for the
inclusion $M\to E_M$ preserves this condition in the second
specialization. For the other chart, the construction factors
through the commutative coefficient diagram
\[
 \begin{array}{ccc}
 M(S\{v\}[1/v])&\xrightarrow{\ \rho_v^M\ }&E_M(S)\\
 \big\downarrow&&\big\downarrow\\
 M(F\{v\}[1/v])&\xrightarrow{\ \rho_v^M\ }&E_M(F).
 \end{array}
\]
Consequently a section from $M(R[1/v])$ has first specialized
value in $E_M(S)$. The second specialization evaluates this
$u$-regular section at $u=0$, leaving only the $v$-torus.
This proves both inclusions in \eqref{sp:chart-images}.

For the final assertion, put $K=\C(z)$ and use the new
henselian normal ring $K\{u,v\}$, which receives $R$.
The element $z-H$ is a unit there because its closed value
is $z-h\in K^\times$. Apply Lemma~\ref{sp:moving} first
in the $v$-coordinate over the coefficient field
$\Frac(K\{u\})$, and then in the $u$-coordinate in the
strong group $E_{\mathcal Q}$. The first closed height is
$H(u,0)$ and the second is $h$; the retained torus coordinate
is unchanged at each step. Outer-height contraction commutes
with adjoining this independent torus coordinate, so these
two applications give \eqref{sp:height-freezing}.
All given sections and equalities descend to finite smooth
height stages, using henselian neighborhoods of the whole zero
section. The coefficient change is henselian, and the specialization
order is fixed throughout as $v$ followed by $u$.
\end{proof}

\section{Triviality of extending torsors on the height-open chart}
\label{inf:section}

Retain the projection and height model of
Propositions~\ref{geom:projection} and~\ref{geom:height-model}.
Thus $\Gamma\subset\A^2$ has normalization $\A^1$, its node set is
$P$, and
\[
 B=\A^2-P,\quad D=\A^2-\Gamma,\quad
 W=(B\times\A^1_w)-(\Gamma^\circ\times0).
\]
The common affine-line-torsor model identifies torsors on
$\bigl(\A^3_{\C}\setminus i(\A^1_{\C})\bigr)^\circ$ and $W$.
The filling morphism will show that every torsor originating on
$\bigl(\A^3_{\C}\setminus i(\A^1_{\C})\bigr)$ is trivial on the full height-open chart.

\begin{lemma}\label{inf:open-line}
For a nonempty open $Z\subset\A^1_\C$ and a strong group ${\mathcal Q}$,
$H^1_{\Nis}(Z,{\mathcal Q})=*$.
\end{lemma}

\begin{proof}
A Nisnevich torsor is trivial at the generic point, since a
Nisnevich covering of a field has a section.  Its generic
trivialization spreads to a nonempty open $E\subset Z$.
The open $Z'=\A^1-(Z-E)$ contains $\A^1-Z$ and satisfies
$Z'\cap Z=E$.  Glue the torsor on $Z$ to the trivial torsor on
$Z'$ along its chosen trivialization on $E$.  Strong invariance
gives $H^1(\A^1,{\mathcal Q})=H^1(\Spec\C,{\mathcal Q})=*$, so this extension, and
therefore the original torsor, is trivial.
\end{proof}

\begin{lemma}
\label{inf:annular-descent}
Let $S$ be smooth affine, $H\in\mathcal O(S)$, and
$A=\mathcal O(S)\{u\}$.  A ${\mathcal Q}$-torsor $\mathcal P$ on $S\times\Gm^{(w)}$
is trivial if its pullback to $\Spec A[1/u]$ by
\[
 (s,u)\longmapsto(s,u^{-1}-H(s))
\]
is trivial.
\end{lemma}

\begin{proof}
By Lemma~\ref{sp:mapping-group}, $\mathcal P$ is an $E_{\mathcal Q}$-torsor on $S$.
Choose an affine Nisnevich cover $S_i$ trivializing it.  Its
transitions are sections
$q_{ij}\in {\mathcal Q}(S_{ij}\times\Gm^{(w)})$; affine refinements of
intersections are understood when necessary.  Pull the assumed
global annular trivialization to each ring
$\mathcal O(S_i)\{u\}[1/u]$.  It supplies gauges satisfying the
pulled transition equations.  Apply the group homomorphism
$\rho_u^{\mathcal Q}$ to these equations.  Naturality and
\eqref{sp:infinity-change} give an ordinary coboundary for
$q_{ij}(s,\xi^{-1})$ on the original Nisnevich cover of $S$.
The inversion automorphism of $\Gm$ therefore pulls $\mathcal P$ back to a
trivial torsor, and $\mathcal P$ itself is trivial.

The construction uses the pullbacks of the given global trivialization
to the henselian tubes.
\end{proof}

\begin{theorem}
\label{inf:height-trivial}
Let ${\mathcal Q}$ be any strongly $\A^1$-invariant group sheaf.
Every ${\mathcal Q}$-torsor on $W$ corresponding to the restriction of a
${\mathcal Q}$-torsor on $\bigl(\A^3_{\C}\setminus i(\A^1_{\C})\bigr)$ is trivial on $B\times\Gm^{(w)}$.
\end{theorem}

\begin{proof}
When $P$ is empty, $B=\A^2$ and the assertion follows from strong
invariance and Lemma~\ref{inf:open-line} with $Z=\Gm$.
Suppose $P$ is nonempty. Use the coordinates of the filling
construction~\ref{geom:filling}, so that
$P=\{(a_i,0)\}$ and the embedding is $(f,g,h)$.
The coordinate $g$ is nonconstant; otherwise the irreducible plane
image would be a smooth line.  The height $h$ is integral over
$\C[g]$, since $\C[t]$ is finite over this ring.  Its monic minimal
polynomial gives an equation $R(y,z)$ for the $(g,h)$-image, monic
in $z$.  The polynomial $R_0(z)=R(0,z)$ is nonzero and is monic
of the same $z$-degree.  Remove its roots and all heights $h(t)$
with $f(t)=0$ from the height line, obtaining a nonempty open $Z$.
The morphism
\begin{equation}\label{inf:filling-map}
 j:\A^2\times Z\longrightarrow \bigl(\A^3_{\C}\setminus i(\A^1_{\C})\bigr),\qquad
 j(a,b,z)=\left(a\frac{R(b,z)}{R_0(z)},b,z\right)
\end{equation}
avoids the curve.  Indeed an intersection would give $R(b,z)=0$,
then $f(t)=0$, and finally an excluded height.  By
Lemma~\ref{inf:open-line} and strong invariance,
$H^1(\A^2\times Z,{\mathcal Q})=*$. Every torsor on the complement thus supplies a
global trivialization after pullback by \eqref{inf:filling-map}.

Put $q(x)=\prod_i(x-a_i)$ and let
\[
 S=\{q(x)\alpha+y\beta=1\}\longrightarrow B
\]
be the affine-line torsor. Use the regular linear
height extension $\widehat h$ of Lemma~\ref{geom:linear-height}.
The height-straightened form of the filling is the roof
\[
 B\times Z\longleftarrow
 E=\{q(a\Lambda(b,z))\alpha+b\beta=1\}\longrightarrow W,
 \qquad \Lambda=R/R_0,
\]
whose right arrow is
\[
 (a,b,z,\alpha,\beta)\longmapsto
 (a\Lambda,b,z-\widehat h(a\Lambda,b,\alpha,\beta)).
\]
The right arrow uses the height extension chosen above.

At infinity put $u=1/z$ and
\[
 \ell(y,u)=R(y,1/u)/R_0(1/u).
\]
Monicity makes $\ell$ regular near $u=0$, with
$\ell(y,0)=1$ and $\ell(0,u)=1$.  It is a unit in
$\mathcal O(S)\{u\}$.  On the annulus over the fixed base $S$,
substitute
\[
 a=x/\ell(y,u),\qquad b=y,\qquad z=u^{-1},
\]
retaining $x,y,\alpha,\beta$.  The equation of $E$ becomes
exactly $q(x)\alpha+y\beta=1$.  The source lies in $B\times Z$:
at $y=0$, $\ell=1$, so the excluded plane points stay excluded;
and every omitted height $c$ is avoided since
$u^{-1}-c=u^{-1}(1-cu)$ is a unit on the annulus.
The target plane coordinates are $(x,y)$, and the
target height is
\[
 w=u^{-1}-\widehat h(x,y,\alpha,\beta).
\]
It is invertible there.  Pulling the global filling trivialization
to this annulus proves the hypothesis of
Lemma~\ref{inf:annular-descent} for the chart torsor on
$S\times\Gm$.  That torsor is therefore trivial.  Finally
$S\to B$ is an affine-line torsor, hence an $\A^1$-equivalence.
Since $BE_{\mathcal Q}$ is local, its pullback is an equivalence on
$E_{\mathcal Q}$-torsors.  Triviality descends to $B\times\Gm$, as required.
\end{proof}

\begin{proposition}
\label{inf:fibers}
Torsors on $W$ with a specified trivialization on
the standard section $B\times\{1\}$ form the discrete parameter
space
\begin{equation}\label{inf:standard-fiber}
 {\mathcal Q_{-2}}(\Gamma^\circ).
\end{equation}
Torsors with a specified trivialization on the entire chart
$B\times\Gm$ instead form the discrete parameter space ${\mathcal Q_{-1}}(D)$.
After retaining a specified fiber at $(d_*,1)$, but forgetting the
height-chart trivialization, their classes are
\begin{equation}\label{inf:pointed-fiber}
 {\mathcal Q_{-1}}(\C)\backslash {\mathcal Q_{-1}}(D)\cong \bigl({\mathcal Q_{-1}}\bigr)_*(D),\qquad
 \bigl({\mathcal Q_{-1}}\bigr)_*(D)=\ker[{\mathcal Q_{-1}}(D)\xrightarrow{d_*^*}{\mathcal Q_{-1}}(\C)],
\end{equation}
with unique representative $c(d_*)^{-1}c$.
\end{proposition}

\begin{proof}
The two open charts of $W$ give the homotopy pushout
\[
 W\simeq(B\times\Gm)\cup^h_{D\times\Gm}D.
\]
The standard section splits the projection to $B$.  Its pointed
quotient is $(B/D)\wedge(\Gm,1)$.  Adjunction and
Lemma~\ref{sp:mapping-group} identify the corresponding mapping
fiber with $\Map_*(B/D,B(\mathcal Q_{-1}))$.  Purity for the smooth principal
divisor $\Gamma^\circ\subset B$, framed by its plane equation,
makes this the discrete space $\mathcal Q_{-2}(\Gamma^\circ)$, proving
\eqref{inf:standard-fiber}.

For the larger trivialized chart, mapping the same pushout into
$B{\mathcal Q}$ gives
\[
 \Tors_{\mathcal Q}(W)\simeq
 \Tors_{E_{\mathcal Q}}(B)\mathbin{\times^h}_{\Tors_{E_{\mathcal Q}}(D)}\Tors_{\mathcal Q}(D).
\]
The homotopy fiber of $B{\mathcal Q}\to BE_{\mathcal Q}$ is the discrete coset sheaf
$E_{\mathcal Q}/{\mathcal Q}={\mathcal Q_{-1}}$.  Its $D$-sections give exactly ${\mathcal Q_{-1}}(D)$.
Changing a height-chart trivialization acts by
$E_{\mathcal Q}(B)={\mathcal Q_{-1}}(B)\rtimes {\mathcal Q}(B)$.  Unramified Hartogs and invariance
give ${\mathcal Q}(B)={\mathcal Q}(\C)$ and ${\mathcal Q_{-1}}(B)={\mathcal Q_{-1}}(\C)$, since $B$ is the affine
plane minus finitely many points.  Requiring the original torsor's
chosen fiber at $(d_*,1)$ to stay fixed removes the ${\mathcal Q}(\C)$
factor: evaluation at height $1$ is the projection to it.
The remaining action is left ${\mathcal Q_{-1}}(\C)$-translation.  Evaluation
at $d_*$ gives its unique normalized representative, proving
\eqref{inf:pointed-fiber}.
\end{proof}

Forgetting the pointing further quotients \eqref{inf:pointed-fiber}
by constant conjugation under ${\mathcal Q}(\C)$. Theorem
\ref{inf:height-trivial} places each complement torsor in this
space of torsors trivial on the height chart.

\section{Finite clutching coordinates}
\label{cl:section}

The height model reduces the remaining problem to sections on the
complement of the projected plane curve. We first compute these sections
for strictly invariant abelian coefficients. The projection fixed in
Proposition~\ref{geom:projection} has only split ordinary nodes. At each
node, two explicit normalization branches give a sum map on residues;
its kernel contributes exactly one additional coefficient. This local
calculation will also fix the signs of the mixed coefficients used in
the nonabelian comparison.

\begin{proposition}
\label{cl:hi}
Let $\Gamma\subset\A^2_\C$ be irreducible, with finite normalization
$\nu:\A^1_t\to\Gamma$, and suppose that its only singularities are
ordinary nodes $P=\{p_1,\ldots,p_r\}$. Write
\[
 \nu^{-1}(p_i)=\{t_i^+,t_i^-\},\qquad
 S_0=\{t_i^+,t_i^-:1\leq i\leq r\},\qquad D=\A^2_\C-\Gamma.
\]
For every strictly invariant abelian sheaf $M$,
\begin{equation}\label{cl:acyclic}
 H^j_{\Nis}(D,M)=0\quad(j>0).
\end{equation}
After choosing a basepoint $d_*\in D(\C)$, a point
$t_0\in\A^1_\C-S_0$, and an ordering of the branches at every node,
there is a functorial isomorphism
\begin{equation}\label{cl:abelian-basis}
 M(D)\cong M(\C)\oplus M_{-1}(\C)\oplus M_{-2}(\C)^r.
\end{equation}
The last $r$ coordinates are the signed differences of the two branch
residues at the corresponding nodes.
\end{proposition}

\begin{proof}
Put $B=\A^2_\C-P$ and
$\Gamma^\circ=\Gamma-P\cong\A^1_\C-S_0$. Purity at the removed
surface points and strict invariance on $\A^2$ give
\[
 M(B)=M(\C),\qquad H^1(B,M)=M_{-2}(\C)^r,
 \qquad H^j(B,M)=0\quad(j\geq2).
\]
On $\Gamma^\circ$ choose the normal frame prescribed by adjunction,
\begin{equation}\label{cl:canonical-frame}
 e=dt\otimes(dx\wedge dy)^{-1}.
\end{equation}
It is regular and nowhere zero on this smooth curve. Purity for the
divisor $\Gamma^\circ\subset B$ yields the localization sequence
\begin{equation}\label{cl:localization}
\begin{split}
 0\longrightarrow M(\C)\longrightarrow M(D)
 \longrightarrow M_{-1}(\Gamma^\circ,e)
 \xrightarrow{\delta}M_{-2}(\C)^r
 \longrightarrow H^1(D,M)\longrightarrow0.
\end{split}
\end{equation}
The punctured affine line has no positive cohomology with these
coefficients: localization on $\A^1$ reduces this to the residue map
onto the finitely many missing points, and the symbols $[t-v]m$
give residue-one lifts. The same localization sequence then shows
$H^j(D,M)=0$ for $j\geq2$.

We compute $\delta$ directly at a node. In a split henselian chart
write the curve as $uv=0$, order its branches as $u=0$ and $v=0$,
and frame the ambient cotangent determinant by $du\wedge dv$.
On the first normalization branch, $v$ is a uniformizer; on the second,
$u$ is a uniformizer. Since the normalization is immersive on each
branch, these parameters differ from $t-t_i^+$ and $t-t_i^-$ by
units with nonzero complex closed values. The residue-field transfers
are identities. The resulting determinant-line factors have classes
$\langle c\rangle=1$ for $c\in\C^\times$.

Thus the degree-two Rost--Schmid differential, with the normal frame
\eqref{cl:canonical-frame} transported to the two branches, is the sum
of their one-dimensional residues. Equivalently, the local map factors as
\[
 M_{-1}(\Gamma^\circ,e)
 \xrightarrow{(\partial_{t-t_i^+},\,\partial_{t-t_i^-})}
 M_{-2}(\C)\oplus M_{-2}(\C)
 \xrightarrow{\ (a,b)\mapsto a+b\ }M_{-2}(\C).
\]
The common sign of the final frame can be fixed once at each point.
Both branches have multiplicity one. This gives
\begin{equation}\label{cl:branch-sum}
 \delta(a)_i=
 \partial_{t-t_i^+}(a)+\partial_{t-t_i^-}(a).
\end{equation}
This is the normalization-branch description of the Rost--Schmid
differential in \cite[Chapter~5]{Morel12}, applied only to the split
nodal chart just described. A vanishing factor in the plane-equation
normal frame is retained as a change of framing; it is not an extra
intersection multiplicity in \eqref{cl:branch-sum}.

For any missing point $v\in S_0$ and $m\in M_{-2}(\C)$, the
normalized symbol
\[
 s_v(m)=([t-v]-[t_0-v])m
       \in M_{-1}(\Gamma^\circ,e)
\]
has residue $m$ at $v$, zero residue at every other missing point,
and value zero at $t_0$. These expressions use the usual contraction
evaluation action: $m$ represents a morphism
$\mathbf K_2^{\mathrm{MW}}\to M$, whose contraction evaluates
the indicated degree-one symbol. They give an explicit splitting
\[
 M_{-1}(\Gamma^\circ,e)
 \cong M_{-1}(\C)\oplus
       \bigoplus_{i=1}^r\bigl(M_{-2}(\C)\oplus M_{-2}(\C)\bigr).
\]
In these coordinates $\delta$ is zero on the first summand and is
the sum on each displayed pair. It is surjective, proving
$H^1(D,M)=0$. Its kernel has the constant coordinate and the $r$
branch differences
\[
 (m_i,-m_i),\qquad 1\leq i\leq r.
\]
Finally evaluation at $d_*$ splits the first $M(\C)$ in
\eqref{cl:localization}. These splittings give
\eqref{cl:abelian-basis}, naturally in $M$. When $r=0$, the same
argument is localization at a smooth affine line and has no branch
difference summands.
\end{proof}

\begin{lemma}
\label{cl:universal-basis}
Choose an irreducible plane equation $F$ and put
$F_*=F/F(d_*)$.  There are fixed sections
$\chi_1,\ldots,\chi_r\in\mathbf K_2^{\mathrm{MW}}(D)$, vanishing at $d_*$, such that
for every strict $M$ the isomorphism
\eqref{cl:abelian-basis} is given by
\begin{equation}\label{cl:basis-map}
 (m_0,m_1,n_1,\ldots,n_r)\longmapsto
 m_0+[F_*]m_1+\sum_{j=1}^r(\chi_j)_*(n_j).
\end{equation}
Here $(\chi_j)_*(n_j)$ is the image of $\chi_j$ under the unique
strict morphism $\mathbf K_2^{\mathrm{MW}}\to M$ represented by $n_j\in M_{-2}(\C)$.
The basis includes all residues of the plane equation.
\end{lemma}

\begin{proof}
At each node choose $t_i^-$ as the negative branch and $t_i^+$
as the positive branch. Take the residue vector $(1,-1)$ there,
zero at every other node, and value zero at $t_0$.
These $r$ vectors form a basis of the branch-sum kernel.
Apply \eqref{cl:localization} to $M=\mathbf K_2^{\mathrm{MW}}$; the desired coefficients
belong to $(\mathbf K_2^{\mathrm{MW}})_{-2}(\C)=\GW(\C)=\Z$.  Exactness supplies lifts
$\chi_j$, and subtracting their values at $d_*$ makes them based
without changing their residues.  Naturality gives the same lifts
for all $M$ by their universal morphisms.

The residue of $[F_*]m_1$ in the frame
\eqref{cl:canonical-frame} is $\langle v(t)\rangle m_1$, where
$v$ is the nonzero normal-frame ratio.  Its value at $t_0$ is $m_1$,
since every nonzero complex scalar is a square. The ratio is a unit
on the smooth normalization with its node preimages removed; its zeros
and poles at the missing points can contribute branch coordinates.
Given a section of $M(D)$, subtract its constant normal coordinate
using the $[F_*]m_1$ term, then its remaining branch-difference
vector using the $\chi_j$, and finally its value at $d_*$. Each
step is unique. This triangular calculation proves
\eqref{cl:basis-map} while retaining all framing contributions.
\end{proof}

Let ${\mathcal Q}$ be strongly $\A^1$-invariant.
The contraction results in Proposition~\ref{con:structure} make
${\mathcal Q_{-1}}$ nilpotent of class at most two and ${\mathcal Q_{-2}},{\mathcal Q_{-3}}$ strictly invariant
abelian. Use the center $Z({\mathcal Q_{-1}})$ and its central quotient $\mathcal Q_{-1}/Z(\mathcal Q_{-1})$.
The center is strongly invariant by the centralizer theorem
\cite[Proposition~3.1.20]{AFH22}. It is abelian and therefore strictly
invariant over the perfect field $\C$. The central quotient is strongly
invariant by \cite[Lemma~3.1.14(3), Proposition~3.1.22]{AFH22}: a
quotient by a central strict subsheaf fits into the simplicial fiber
sequence with outer terms $B(\mathcal Q_{-1})$ and
$K(Z(\mathcal Q_{-1}),2)$, which are both $\A^1$-local.
Since $\mathcal Q_{-1}$ has class at most two, this particular quotient
is abelian, and hence it too is strictly invariant.
The centralizer and central-quotient theorems therefore allow us to
apply the preceding abelian calculation to these two sheaves.

\begin{proposition}
\label{cl:finite}
Fix an order on the sections $\chi_j$ of
Lemma~\ref{cl:universal-basis}. Multiplication of sections
defines a natural bijection of sets
\begin{equation}\label{cl:finite-map}
\begin{split}
 {\mathcal Q_{-2}}(\C)\times {\mathcal Q_{-3}}(\C)^r&\xrightarrow{\sim}\bigl({\mathcal Q_{-1}}\bigr)_*(D),\\
 (t,\lambda_1,\ldots,\lambda_r)&\longmapsto
 e_{F_*}(t)h_{\lambda_1}(\chi_1)\cdots
 h_{\lambda_r}(\chi_r).
\end{split}
\end{equation}
Here $h_\lambda:\mathbf K_2^{\mathrm{MW}}\to {\mathcal Q_{-1}}$ is the homomorphism of Lemma~\ref{con:square}
represented by $\lambda\in {\mathcal Q_{-3}}(\C)$.
This is a statement about sets; coefficientwise additivity of
\eqref{cl:finite-map} is not asserted.
\end{proposition}

\begin{proof}
Proposition~\ref{cl:hi} gives $H^1(D,Z({\mathcal Q_{-1}}))=H^1(D,\bigl({\mathcal Q_{-1}}/Z({\mathcal Q_{-1}})\bigr))=0$.
The central extension
\[
 1\longrightarrow Z({\mathcal Q_{-1}})\longrightarrow {\mathcal Q_{-1}}\longrightarrow \bigl({\mathcal Q_{-1}}/Z({\mathcal Q_{-1}})\bigr)
 \longrightarrow1
\]
therefore yields $H^1(D,{\mathcal Q_{-1}})=*$ and a surjection
${\mathcal Q_{-1}}(D)\to \bigl({\mathcal Q_{-1}}/Z({\mathcal Q_{-1}})\bigr)(D)$. Passing to based sections gives the
central extension
\begin{equation}\label{cl:central-based}
 1\longrightarrow (Z({\mathcal Q_{-1}}))_*(D)\longrightarrow \bigl({\mathcal Q_{-1}}\bigr)_*(D)
 \longrightarrow(\bigl({\mathcal Q_{-1}}/Z({\mathcal Q_{-1}})\bigr))_*(D)\longrightarrow1.
\end{equation}
Indeed an arbitrary lift of a based section of $\bigl({\mathcal Q_{-1}}/Z({\mathcal Q_{-1}})\bigr)$ has value
in $Z({\mathcal Q_{-1}})(\C)$ at $d_*$, which can be removed by a constant central
section.

Decompose the image of $c\in \bigl({\mathcal Q_{-1}}\bigr)_*(D)$ in $(\bigl({\mathcal Q_{-1}}/Z({\mathcal Q_{-1}})\bigr))_*(D)$ by
\eqref{cl:basis-map}.  Exactness of contraction, followed by the
splitting of Nisnevich covers over $\Spec\C$, lifts its coefficients
through ${\mathcal Q_{-2}}(\C)$ and ${\mathcal Q_{-3}}(\C)$.  The resulting ordered product in
\eqref{cl:finite-map} has the same image as $c$.  Their discrepancy
belongs to $(Z({\mathcal Q_{-1}}))_*(D)$, whose abelian basis supplies central
corrections to those coefficients.

These corrections can be absorbed into the product. For
$z\in(Z({\mathcal Q_{-1}}))_{-2}(\C)$ one has
\begin{equation}\label{cl:central-correction}
 h_{\lambda+z}=h_\lambda h_z.
\end{equation}
The image of $h_z$ is central, so the right-hand side is a
homomorphism out of $\mathbf K_2^{\mathrm{MW}}$.  It agrees with the left-hand side on
the universal pointed square; the square universal property
Lemma~\ref{con:square} proves the identity.  The analogous
correction in $e_{F_*}$ is an ordinary evaluation homomorphism.
This proves surjectivity.

If two ordered products agree, their coefficient images in
$(\bigl({\mathcal Q_{-1}}/Z({\mathcal Q_{-1}})\bigr))_{-1}(\C)$ and $(\bigl({\mathcal Q_{-1}}/Z({\mathcal Q_{-1}})\bigr))_{-2}(\C)$ agree by the abelian basis.
Their differences consequently come from $(Z({\mathcal Q_{-1}}))_{-1}(\C)$ and
$(Z({\mathcal Q_{-1}}))_{-2}(\C)$.  Identity~\eqref{cl:central-correction} moves
all discrepancies into a single central product.  The abelian
basis for $Z({\mathcal Q_{-1}})$ then makes each difference zero.  This proves
injectivity.

Although the centers used in this proof need not be functorial
under arbitrary homomorphisms, the displayed forward formula is:
evaluation commutes with homomorphisms, and square universality
sends $h_\lambda$ to $h_{\lambda'}$.  Its bijective inverse is
therefore natural for every homomorphism of strong targets.
\end{proof}

Combining Theorem~\ref{inf:height-trivial} with
Propositions~\ref{inf:fibers} and~\ref{cl:finite}, every pointed
${\mathcal Q}$-torsor originating on $\bigl(\A^3_{\C}\setminus i(\A^1_{\C})\bigr)$ has a unique tuple
\begin{equation}\label{cl:actual-tuple}
 (t,\lambda_1,\ldots,\lambda_r)\in {\mathcal Q_{-2}}(\C)\times {\mathcal Q_{-3}}(\C)^r.
\end{equation}
The next lemma fixes the reference tuple against which its
remaining entries must be compared.

\begin{lemma}
\label{cl:comparison}
Let $\phi:\bigl(\A^3_{\C}\setminus i(\A^1_{\C})\bigr)\to\A^2-0$ be the complete-intersection normal map
of Proposition~\ref{geom:normal-map}.  Its universal $\mathbf K_2^{\mathrm{MW}}$-torsor has first coordinate
$1\in\GW(\C)=\Z$ in \eqref{cl:finite-map}.  Consequently, for
every $t\in {\mathcal Q_{-2}}(\C)$, the actual comparison torsor obtained by
pushing forward along $h_t:\mathbf K_2^{\mathrm{MW}}\to {\mathcal Q}$ has coordinates
\begin{equation}\label{cl:comparison-tuple}
 (t,\lambda_1^{\mathrm{ext}}(t),\ldots,
       \lambda_r^{\mathrm{ext}}(t)).
\end{equation}
Its first coordinate is exactly $t$.  No assertion that its other
coordinates vanish, or are additive in $t$, is required.
\end{lemma}

\begin{proof}
First take the abelian target $\mathbf K_2^{\mathrm{MW}}$.  The globally framed normal
parameter of the normal map is the rank-one degree $1$.  At the
chosen smooth normalization point $t_0$, the determinant comparing
its global normal frame to the height frame is a nonzero complex
scalar, hence has Grothendieck--Witt form $1$.  Thus its constant
normal coordinate in the localization calculation is $1$.
The $F_*$ term has constant normal coordinate $1$ and every fixed
$\chi_j$ has constant normal coordinate zero.  The triangular
abelian basis therefore makes the first finite coordinate of this
$\mathbf K_2^{\mathrm{MW}}$-torsor equal to $1$.

Now push this one torsor forward under $h_t$.  Naturality of the
set bijection~\eqref{cl:finite-map} carries its first coordinate
$1$ to $t$.  This proves \eqref{cl:comparison-tuple}.  This
uses the abelian calculation for the universal comparison and then
naturality.
\end{proof}

\begin{definition}\label{cl:defects}
For the tuple~\eqref{cl:actual-tuple}, define its defects by
\[
 b_j=\lambda_j-\lambda_j^{\mathrm{ext}}(t)\in {\mathcal Q_{-3}}(\C),
 \qquad 1\le j\le r.
\]
The differences are formed in the strictly invariant abelian
group ${\mathcal Q_{-3}}(\C)$. The reference coordinates are those of the
normal comparison torsor, including its framing contributions.
\end{definition}


\section{The node equation}\label{sec:node-equation}

As in a knot diagram, projection simplifies the curve while the
lift records the separation of its branches. At each ordinary node
the two complex heights have a nonzero difference $\Delta$.
The equations below retain this difference and the full chart gauges
in torsor clutching under a strongly $\A^1$-invariant group sheaf.
Rescaling the heights and discarding the pointing correction could
introduce a ramified gauge, so we use the original heights throughout.

\begin{lemma}\label{node:normal-form}
Every $c\in {\mathcal Q_{-1}}(\Gm^{(u)}\times\Gm^{(v)})$ satisfying $c(1,1)=1$ has a unique
factorization
\[
 c=e_u(p)e_v(q)a_\ell,\qquad
 p,q\in {\mathcal Q_{-2}}(\C),\quad\ell\in {\mathcal Q_{-3}}(\C),
\]
where $a_\ell$ is fully pointed in $u,v$. The mixed factor commutes with
the two single-coordinate factors.
\end{lemma}
\begin{proof}
Take the two single-coordinate restrictions $A=c(u,1)$ and $B=c(1,v)$.
They are respectively $e_u(p)$ and $e_v(q)$. The remainder
$B^{-1}A^{-1}c$ restricts to $1$ on both coordinate sections, hence is
the fully pointed ${\mathcal Q_{-1}}$-family represented by $\ell\in {\mathcal Q_{-3}}(\C)$.
Evaluations recover $A,B$ and then the remainder, proving uniqueness.
The commutations are Lemma~\ref{con:mixed}.
\end{proof}

\begin{lemma}\label{node:height-symbol}
For $\Delta\in \C^*$ and a variable $z$,
\begin{equation}\label{eq:height-symbol}
 [z][z-\Delta]-[z]^2=[\Delta]([z-\Delta]-[z])
 \quad\text{in }K_2^{\mathrm{MW}}(\C(z)).
\end{equation}
Also $[z][z-\Delta]=[z][-1]+[\Delta]([z-\Delta]-[z])$.
\end{lemma}
\begin{proof}
Put $s=z/\Delta$. Since $\Delta$ and $-1$ are squares in $\C$,
$\langle\Delta\rangle=\langle-1\rangle=1$ and multiplication by these
constant units has the additive symbol formula. Thus
$[z]=[\Delta]+[s]$ and $[z-\Delta]=[\Delta]+[s-1]$.
The constant square $[\Delta]^2$ is zero, and the two cross terms in
$([\Delta]+[s])^2$ cancel by graded anticommutativity. Finally Steinberg
and $[s-1]=[-1]+[1-s]$ give $[s][s-1]=[s][-1]=[s]^2$.
Expanding the two sides proves the first equality. The second follows
from $[z]^2=[z][-1]$. The calculation is in Milnor--Witt theory itself,
not just in its Milnor quotient.
\end{proof}

For $p,q\in {\mathcal Q_{-2}}(\C)$, use the meridian family convention
$P_p(u,w)=e_u(p)(w)$. Define
\begin{equation}\label{eq:DDelta}
 D_\Delta(p,q)(u,v,w)
   =[P_p(u,\Delta),P_q(v,w)].
\end{equation}
It is fully pointed in $u,v,w$ and hence represents an element of ${\mathcal Q_{-3}}(\C)$.
By Proposition~\ref{con:bracket}, it is biadditive in $p,q$ and is the
evaluation of one coordinate of $\beta_{2,2}(p,q)$ at $\Delta$. We denote
this, with the fixed coordinate order understood, by
$D_\Delta(p,q)=[\Delta]\beta_{2,2}(p,q)$. If slots are reordered, the
same formulas acquire the corresponding fixed signs.

\begin{lemma}\label{node:collection}
Let $c=e_u(p)e_v(q)a_\ell$ be as in
Lemma~\ref{node:normal-form}, and set $w_1=z$, $w_2=z-\Delta$.
Write
\[
 A_i=P_p(u,w_i),\quad B_i=P_q(v,w_i),\quad
 \mathcal L_i=a_\ell(u,v)(w_i),\quad
 a=A_2A_1^{-1},\quad b=B_2B_1^{-1}.
\]
The left-transition $g=c(w_2)c(w_1)^{-1}$ satisfies
\begin{equation}\label{eq:collected-transition}
 a^{-1}gb^{-1}
 =[A_1,B_2][A_1,B_1]^{-1}\mathcal L_2\mathcal L_1^{-1}.
\end{equation}
Its fully pointed normal parameter is
\begin{equation}\label{eq:node-parameter}
 \ev_{z-\Delta}\bigl(\ell+D_\Delta(p,q)\bigr)
 -\ev_z\bigl(\ell+D_\Delta(p,q)\bigr)
 \quad\text{in }{\mathcal Q_{-2}}(\C(z)).
\end{equation}
\end{lemma}
\begin{proof}
We first justify the commutations used in collecting factors. Any two
$\mathcal L$-factors share the two plane slots $u,v$; their commutator vanishes by
two repetitions, regardless of their heights. An $\mathcal L$-factor and an
$A$-factor share the $u$ slot. If their heights agree, the shared height
is a second repetition and kills their commutator. At the two distinct
heights, make both heights independent first. The repeated plane slot
produces one $-1$ slot. Lemma~\ref{node:height-symbol} expands the two
height symbols into terms with either another $-1$ slot or a constant
$\Delta$ slot. Each vanishes because $[-1]^2=[\Delta][-1]=0$.
The $\mathcal L$--$B$ case uses the $v$ slot in the same manner.

The commutators $[A_i,B_j]$ commute with the factors involved in
\eqref{eq:collected-transition}. For an $A$-factor at the same height
as $A_i$, its outer conjugation repeats the plane and height slots.
For the other height, use the same height expansion after repeating the
plane slot. The $B$ case is identical. This proves the required
commutations before any informal class-two collection in ${\mathcal Q}$ is made;
the whole group ${\mathcal Q}$ need not have class two.

With these commutations justified,
\[
 a^{-1}gb^{-1}
 =A_1(B_2B_1^{-1})A_1^{-1}(B_2B_1^{-1})^{-1}\mathcal L_2\mathcal L_1^{-1},
\]
which is \eqref{eq:collected-transition}. The independent-height family
$[P_p(u,w_1),P_q(v,w_2)]$ is pointed in both height variables. Apply
Lemma~\ref{con:square} in those two variables and then
\eqref{eq:height-symbol}. The commutator difference in
\eqref{eq:collected-transition} becomes
\[
 \ev_{z-\Delta}D_\Delta(p,q)
 \bigl(\ev_zD_\Delta(p,q)\bigr)^{-1}.
\]
The mixed $\mathcal L$-difference supplies the other two terms. These fully
pointed two-plane-coordinate sections belong to the abelian group
${\mathcal Q_{-2}}(\C(z))$, giving \eqref{eq:node-parameter}.
\end{proof}

\begin{proposition}\label{node:axis-equation}
At each node, let $c$ be the normalized full torus specialization of the
global clutching section of a torsor coming from $\bigl(\A^3_{\C}\setminus i(\A^1_{\C})\bigr)$, and write it as in
Lemma~\ref{node:normal-form}. Let $h_1$ be the closed height on $u=0$
and $h_2$ the closed height on $v=0$, and put $\Delta=h_2-h_1\ne0$.
Then
\begin{equation}\label{eq:axis-equation}
 \ell+D_\Delta(p,q)=0\quad\text{in }{\mathcal Q_{-3}}(\C).
\end{equation}
\end{proposition}
\begin{proof}
By Lemma~\ref{geom:node-charts}, the actual ordered transition is
$c(z-H_u)c(z-H_v)^{-1}$, with $H_u(0)=h_2$ and $H_v(0)=h_1$.
Pass to the henselian normal ring $\C(z)\{u,v\}$; this is a new
henselization over the generic axis field, receiving the original
chart. Full two-coordinate specialization, with moving-height
compatibility, sends this transition to
$c(z-h_2)c(z-h_1)^{-1}$ and sends the two chart gauge groups to
their respective single-torus gauge groups. The torsor extends over the
generic axis and is trivial over its henselian local normal scheme.
Its trivializing gauges therefore specialize to a coboundary
for this torus transition.

The plane-constant factor $c_0=c(1,1)$ can be left-removed. The two
resulting correcting factors $c_0(z-h_2)^{-1}$ and $c_0(z-h_1)$ are
regular plane-constant gauges; they need not be equal or central in
${\mathcal Q}(\C(z))$. Thus the removal preserves triviality of the axis torsor.
Translate the axis coordinate by $h_1$, without rescaling it, and apply
Lemma~\ref{node:collection}.

The standard punctured plane $\bigl(\A^2_{\C}\setminus\{0\}\bigr)_{\C(z)}$ has torsor classes
${\mathcal Q_{-2}}(\C(z))/{\mathcal Q}(\C(z))$, with constant conjugation action. This follows
directly from $\bigl(\A^2_{\C}\setminus\{0\}\bigr)\simeq\susp\Gm^{\wedge2}$ and forgetting a pointing.
In particular its zero torsor has only the zero normal parameter.
Hence \eqref{eq:node-parameter} is zero. Its residue at $z=0$ is
$-\ell-D_\Delta(p,q)$: the first summand is regular there because
$z-\Delta$ is a unit. This proves \eqref{eq:axis-equation}.

The specialization order remains the one chosen above.
\end{proof}

\section{Vanishing of node defects}\label{sec:defects}

Fix a strongly invariant sheaf ${\mathcal Q}$ and use the finite set coordinates
of Proposition~\ref{cl:finite}. The normal comparison torsor with
first coefficient $t\in {\mathcal Q_{-2}}(\C)$ has coordinates
$(t,\lambda_1^{\mathrm{ext}}(t),\ldots,
\lambda_r^{\mathrm{ext}}(t))$. For an extending torsor with the same first
coefficient define
\begin{equation}\label{eq:defect-definition}
 b_j=\lambda_j-\lambda_j^{\mathrm{ext}}(t)\in {\mathcal Q_{-3}}(\C).
\end{equation}
The comparison is the normal torsor pulled back to the complement, including its higher
coordinates.

\begin{lemma}\label{def:local-additivity}
For a split node $i$, the full specialization of each chosen source
section has the form
\begin{equation}\label{eq:chi-expansion}
 \tau_i\chi_j=\sigma_{ij}+[u]a_{ij}+[v]d_{ij}
                         +m_{ij}[u][v],
\end{equation}
with $\sigma_{ij}\in K_2^{\mathrm{MW}}(\C)$,
$a_{ij},d_{ij}\in \C^*$ and $m_{ij}\in\GW(\C)=\Z$.
For these particular sections, the evaluation
$b\mapsto h_b(\tau_i\chi_j)$ is additive in $b\in {\mathcal Q_{-3}}(\C)$.
The constant and the two single-slot factors have central image in ${\mathcal Q_{-1}}$.
\end{lemma}
\begin{proof}
The usual split evaluation decomposition of a strictly invariant abelian
sheaf on $\Gm^2$, applied to $\mathbf K_2^{\mathrm{MW}}$, gives
$K_2^{\mathrm{MW}}(\C)\oplus K_1^{\mathrm{MW}}(\C)^2\oplus\GW(\C)$
and the displayed expression. Naturality of full group specialization
gives $\tau_{\mathcal Q_{-1}} h_b(\chi_j)=h_b(\tau_{\mathbf K_2^{\mathrm{MW}}}\chi_j)$.

Every source summand except the last is $2$-divisible. Indeed
$K_2^{\mathrm{MW}}(\C)$ is $2$-divisible, and a constant $a=d^2$ has
$[a]=2[d]$: the unit-product formula and
$[d]^2=[d][-1]$ make its extra term zero because
$\eta[-1]=\langle-1\rangle-1=0$.
Their images under a homomorphism to ${\mathcal Q_{-1}}$ are squares, and hence central
by Proposition~\ref{con:structure}. A pure two-symbol evaluates
additively in $b$, by the definition of the group law on ${\mathcal Q_{-3}}$.
For the constant $\sigma_{ij}$ choose a finite decomposition into pure
symbols with constant slots; each has central image, so the additive
evaluation identities can be collected without a commutator error.
The last summand evaluates to the mixed torus family with
coefficient $m_{ij}b$, and this is additive by the contraction group law.
It commutes with the other factors by Lemma~\ref{con:mixed}.
This proves the local assertion.
\end{proof}

\begin{lemma}\label{def:mixed-matrix}
The branch orientations can be chosen so that $m_{ij}=\delta_{ij}$ in
\eqref{eq:chi-expansion}. Equivalently, before the orientation choice,
the mixed matrix is a signed identity matrix.
\end{lemma}
\begin{proof}
Use the split henselian coordinates $u,v$ at node $i$, and order the
two normalization branches as $u=0$ and $v=0$. For the ordered ambient
frame $du\wedge dv$, the double tame residues of the mixed symbol are
\[
 \begin{array}{c|c|c}
  \text{first branch}&\text{first residue of }[u][v]
                    &\text{residue at the node}\\ \hline
  u=0&[v]&1\\
  v=0&-[u]&-1.
 \end{array}
\]
Indeed $[u][v]=-[v][u]$ in these conventions because
$\langle-1\rangle=1$ over $\C$; a uniformizer has residue $1$.
Constant and single-slot terms in \eqref{eq:chi-expansion} have
zero iterated residue. The mixed coefficient can therefore be read
off by first taking the $u$-residue and then the $v$-residue.
On strict abelian coefficients, the full specialization of
Proposition~\ref{sp:double} retains this ordered framed residue:
its evaluation representative is the same pointed torus family.

By construction in Lemma~\ref{cl:universal-basis}, $\chi_j$ has
the branch-residue vector $(1,-1)$ at node $j$ and $(0,0)$ at
every other node. The displayed calculation therefore gives the
Kronecker delta for $m_{ij}$ after passing to ordinary Milnor
$K$-theory. This passage does not lose an integer coefficient:
\[
 \GW(\C)\xrightarrow{\ \operatorname{rank}\ }K_0^{\mathrm M}(\C)
       =\Z
\]
is an isomorphism. Transport from the normalization parameters
$t-t_i^\pm$ to $v,u$ and from $dx\wedge dy$ to $du\wedge dv$
contributes only the closed determinant factors $\langle c\rangle$
with $c\in\C^\times$, all equal to $1$. The possible ordering
sign has already been included in the two branch orientations.
Consequently the equality is integral in $\GW(\C)$ itself.
\end{proof}

\begin{proposition}\label{def:coordinates}
Write the normalized local torus coordinates of the extending torsor as
$(p_i,q_i,\ell_i)$ and those of its comparison as
$(p_i^0,q_i^0,\ell_i^0)$. With the orientation of
Lemma~\ref{def:mixed-matrix},
\begin{equation}\label{eq:local-defects}
 \begin{aligned}
 p_i&=p_i^0+A_i(b),& A_i(b)&=\sum_j[a_{ij}]b_j,\\
 q_i&=q_i^0+A'_i(b),& A'_i(b)&=\sum_j[d_{ij}]b_j,\\
 \ell_i&=\ell_i^0+b_i.
 \end{aligned}
\end{equation}
\end{proposition}
\begin{proof}
Apply $\tau_i$ to the ordered finite factorization of both global
sections. Lemma~\ref{def:local-additivity} replaces every
$h_{\lambda_j}(\tau_i\chi_j)$ by
$h_{\lambda_j^0}(\tau_i\chi_j)h_{b_j}(\tau_i\chi_j)$.
Move a defect factor only past later $\chi$-factors. Its constant and
single-slot pieces are central in ${\mathcal Q_{-1}}$; its mixed piece commutes with
their single-slot and mixed pieces by two repeated coordinates.
Consequently the resulting local section is exactly the comparison
section times the product of the defect factors; the constant of the
plane-equation factor has stayed in place.

Write these two local factors as $c_0E$ and $d_0D$, where $E,D$ are based
at $(1,1)$. The defect constant $d_0$ is central by
Lemma~\ref{def:local-additivity}. The constant of their product is
$c_0d_0$, and left normalization gives
$(c_0d_0)^{-1}c_0E d_0D=ED$.
The single-slot pieces of $D$ are central and the mixed piece commutes
with the based factors of $E$ by Lemma~\ref{con:mixed}. Thus their
normal-form coordinates add exactly. A $u$-only defect coefficient is
the evaluation $[a_{ij}]b_j\in {\mathcal Q_{-2}}(\C)$; the $v$-only one is
$[d_{ij}]b_j$. The mixed coefficient is
$\sum_jm_{ij}b_j=b_i$. This proves all of
\eqref{eq:local-defects}, including the constant-gauge bookkeeping.
\end{proof}

\begin{proposition}\label{def:system}
The defect vector satisfies
\begin{equation}\label{eq:closed-system}
 \begin{split}
 b_i=-[\Delta_i]\bigl(&\beta_{2,2}(p_i^0,A'_i(b))
                  +\beta_{2,2}(A_i(b),q_i^0)\\
                  &+\beta_{2,2}(A_i(b),A'_i(b))\bigr).
 \end{split}
\end{equation}
Fixed slot-order signs may be absorbed in its entries.
\end{proposition}
\begin{proof}
Both extending torsors satisfy
Proposition~\ref{node:axis-equation} with the same original height
difference $\Delta_i$. Subtract their two equations in the abelian
group ${\mathcal Q_{-3}}(\C)$, substitute \eqref{eq:local-defects}, and use the
biadditivity of $D_{\Delta_i}=[\Delta_i]\beta_{2,2}$.
This gives \eqref{eq:closed-system}, with all coefficients and defects
computed in the same global parameterization.
\end{proof}

The equations have a common feature: the equation
at node $i$ has an outer factor $[\Delta_i]$. It is therefore enough
to filter by the height differences alone. The other constants may
act on a filtered coefficient without being added to the generating
set of the filtration.

\begin{lemma}\label{def:filtration}
For the $r$ nonzero height differences $\Delta_1,\ldots,\Delta_r$
and $n\geq2$, put
\[
 J^m\mathcal Q_{-n}(\C)=
 \sum_{i_1,\ldots,i_m=1}^r
 [\Delta_{i_1}]\cdots[\Delta_{i_m}]
 \mathcal Q_{-(n+m)}(\C),
 \qquad J^0\mathcal Q_{-n}(\C)=\mathcal Q_{-n}(\C).
\]
This descending filtration satisfies $J^{r+1}\mathcal Q_{-n}(\C)=0$
for every $n\geq2$. Evaluation at any other constant preserves its
level, with the contraction shift:
\begin{equation}\label{eq:height-filtration-evaluation}
 [a]\bigl(J^m\mathcal Q_{-(n+1)}(\C)\bigr)
       \subseteq J^m\mathcal Q_{-n}(\C)
       \qquad(a\in\C^\times).
\end{equation}
Moreover, for $a,b\geq2$,
\[
 \beta_{a,b}\bigl(J^j\mathcal Q_{-a}(\C),
                  J^s\mathcal Q_{-b}(\C)\bigr)
 \subseteq J^{j+s}\mathcal Q_{-(a+b)}(\C).
\]
\end{lemma}
\begin{proof}
Evaluating one height slot first proves $J^{m+1}\subseteq J^m$.
For a generator in the left side of
\eqref{eq:height-filtration-evaluation}, the constant-symbol action
of Remark~\ref{con:actions} gives
\[
 [a][\Delta_{i_1}]\cdots[\Delta_{i_m}]x
   =(-1)^m[\Delta_{i_1}]\cdots[\Delta_{i_m}]([a]x).
\]
The equality uses graded anticommutativity of constant symbols in
$K_*^{\mathrm{MW}}(\C)=K_*^{\mathrm M}(\C)$ in positive degrees.
Its right side belongs to the stated level. The bracket property
follows by introducing the additional pointed coordinates before
forming the commutator and then applying the evaluation compatibility
and biadditivity of Proposition~\ref{con:bracket}.

Each $[\Delta_i]^2$ is zero by Lemma~\ref{con:constants}.
Every product of $r+1$ of the $r$ height symbols repeats one of
them. Graded anticommutativity brings the repetition together, and
the product is zero. Its action on the strict abelian contraction
tower is therefore zero by Remark~\ref{con:actions}.
For $r=0$, $J^1=0$ and there are no defect coordinates.
\end{proof}

\begin{theorem}\label{def:vanishing}
Every corrected defect $b_j$ of a complement torsor is zero.
For a projection with $r$ nodes, the fixed system
\eqref{eq:closed-system} forces this after at most $r+1$
successive substitutions.
\end{theorem}
\begin{proof}
Suppose every $b_j$ belongs to $J^m\mathcal Q_{-3}(\C)$.
Equation~\eqref{eq:height-filtration-evaluation} gives
\[
 A_i(b),A'_i(b)\in J^m\mathcal Q_{-2}(\C),
\]
regardless of the values of $a_{ij}$ and $d_{ij}$. The first
two bracket terms on the right of \eqref{eq:closed-system}
are in $J^m\mathcal Q_{-4}(\C)$, and the quadratic term is in
$J^{2m}\mathcal Q_{-4}(\C)$. Multiplication by its outer
$[\Delta_i]$ places the linear terms in
$J^{m+1}\mathcal Q_{-3}(\C)$ and the quadratic term in
\[
 J^{2m+1}\mathcal Q_{-3}(\C)
       \subseteq J^{m+1}\mathcal Q_{-3}(\C).
\]
Hence the whole vector moves from level $m$ to level $m+1$.
Start at $m=0$ and repeat $r+1$ times. Lemma~\ref{def:filtration}
then makes every entry zero. The case $r=0$ has no entries.
\end{proof}

\subsection{The case of one node}
\label{subsec:one-node-calculation}

The one-node case shows explicitly how both the arbitrary height
difference and the comparison torsor enter the argument.
Let the two normalization points be $t^+,t^-$, with distinct
embedded heights $h_1,h_2$, and put $\Delta=h_2-h_1$.
Choose the single section $\chi\in\mathbf K_2^{\mathrm{MW}}(D)$
with branch residues $(1,-1)$ and value zero at $d_*$. Its full
specialization is
\[
 \tau\chi=\sigma+[u][a]+[v][d]+[u][v],
 \qquad \sigma\in K_2^{\mathrm{MW}}(\C),\quad a,d\in\C^\times.
\]
Here $[u][a]$ and $[v][d]$ spell out the single-slot notation
of \eqref{eq:chi-expansion}. An extending torsor has finite
coordinates $(t,\lambda)$ and the normal comparison has
coordinates $(t,\lambda^0(t))$. Set
$b=\lambda-\lambda^0(t)\in\mathcal Q_{-3}(\C)$.
If the based local coordinates of the comparison are
$(p^0,q^0,\ell^0)$, Proposition~\ref{def:coordinates} gives
exactly
\[
 p=p^0+[a]b,\qquad q=q^0+[d]b,\qquad
 \ell=\ell^0+b.
\]
Both torsors extend, so their node equations are
\[
 \ell+[\Delta]\beta_{2,2}(p,q)=0,\qquad
 \ell^0+[\Delta]\beta_{2,2}(p^0,q^0)=0.
\]
Subtracting these equalities yields the one-node equation
\begin{equation}\label{eq:one-node-defect}
 b=-[\Delta]\bigl(
   \beta_{2,2}(p^0,[d]b)
  +\beta_{2,2}([a]b,q^0)
  +\beta_{2,2}([a]b,[d]b)\bigr).
\end{equation}
The first substitution writes $b=[\Delta]c$, where the
coefficient in $\mathcal Q_{-4}(\C)$ is
\[
 c=-\beta_{2,2}(p^0,[d]b)
    -\beta_{2,2}([a]b,q^0)
    -\beta_{2,2}([a]b,[d]b).
\]
Since $[a]b=-[\Delta]([a]c)$ and
$[d]b=-[\Delta]([d]c)$, substitution once more into
\eqref{eq:one-node-defect} makes each linear term a value of a
product containing $[\Delta]^2$ and the quadratic term a value
of a product containing $[\Delta]^3$. More explicitly,
evaluation compatibility of the bracket gives the three equalities
\[
\begin{aligned}
 [\Delta]\beta_{2,2}(p^0,[d][\Delta]c)&=0,\\
 [\Delta]\beta_{2,2}([a][\Delta]c,q^0)&=0,\\
 [\Delta]\beta_{2,2}([a][\Delta]c,[d][\Delta]c)&=0.
\end{aligned}
\]
The coordinate-permutation signs in pulling out the constant
symbols do not affect these zero values. Thus the second
substitution gives $b=0$. The original torsor and its
normal comparison have the same finite coordinates and therefore
agree.

\begin{theorem}\label{thm:fundamental}
Let $i:\A^1_{\C}\hookrightarrow\A^3_{\C}$ be a closed embedding.
For any choice of two scheme-theoretic ideal generators $F_1,F_2$, the
normal map $(F_1,F_2)$ induces an isomorphism of Nisnevich sheaves
\[
 \piA{1}\bigl(\A^3_{\C}\setminus i(\A^1_{\C})\bigr)
 \xrightarrow{\ \sim\ }\mathbf K_2^{\mathrm{MW}}.
\]
Equivalently, for every strongly $\A^1$-invariant sheaf of groups
$\mathcal Q$, pullback along this normal map induces a bijection of
pointed Nisnevich $\mathcal Q$-torsor classes from the punctured affine
plane to the given complement.
\end{theorem}

\begin{proof}
Every pointed $\mathcal Q$-torsor on
$\A^3_{\C}\setminus i(\A^1_{\C})$ is trivial on the height-open
chart by Theorem~\ref{inf:height-trivial}. Give its height model the
unique finite coordinates $(t,\lambda_1,\ldots,\lambda_r)$ of
Proposition~\ref{cl:finite}. The pullback of the torsor on
$\A^2_{\C}\setminus\{0\}$ represented by $t$ has the same first
coordinate, by Lemma~\ref{cl:comparison}. Theorem~\ref{def:vanishing}
makes all corrected defects zero. The two pointed torsors on the height model
therefore agree. Proposition~\ref{geom:height-model} first identifies
them on the complement with the vertical node axes removed, and its
fully faithful restriction then identifies them on the complement.

Apply this assertion to
$\mathcal Q=\piA{1}(\A^3_{\C}\setminus i(\A^1_{\C}))$ and its
universal pointed torsor. The identity homomorphism of this sheaf factors
as the homomorphism induced by $(F_1,F_2)$ followed by a homomorphism
$\mathbf K_2^{\mathrm{MW}}\to
\piA{1}(\A^3_{\C}\setminus i(\A^1_{\C}))$.
Hence the kernel of the first homomorphism is trivial. It is already a
split epimorphism by Proposition~\ref{geom:normal-map}, so it is an
isomorphism. If one forgets a pointing at an intermediate stage, the
identity is instead conjugate to this factorization; conjugation has
trivial kernel and gives the same conclusion. The assertion for pointed
torsors follows by applying
$\Hom(-,\mathcal Q)$.
\end{proof}

\part{Concrete rectifications}
\section{Coordinates and concrete rectifications}\label{sec:rectifications}

We construct polynomial rectifiers for several families, together with
inverses defined on the whole ambient space, including the exceptional
fibres of the modifications.

\subsection{Parameter lifts and graph subtraction}

Every closed embedding admits a parameter lift
$H\in\mathbb C[x,y,z]$ with $H(i(t))=t$, since its map on coordinate
rings is surjective. We seek such a lift with a companion $W$ for which
$(x,H,W)$ is a polynomial automorphism. Changing a lift by an element
of the curve ideal is the strategy of \cite[Proposition~3]{Craighero}.
The following graph subtraction works over every commutative coefficient
ring.

\begin{lemma}\label{lem:graph-subtraction}
Suppose $i(t)=(b(t),c(t),d(t))$ and
\[
 (x,H,W)^{-1}=(x,R,S),\qquad H(i(t))=t.
\]
Put $w(t)=W(i(t))$.  Then
\begin{equation}\label{eq:graph-rectifier}
 \Phi(x,y,z)=\bigl(H,\ x-b(H),\ W-w(H)\bigr)
\end{equation}
is a polynomial automorphism satisfying $\Phi(i(t))=(t,0,0)$.
On input $(T,X,Z)$ its inverse is
\begin{equation}\label{eq:graph-inverse}
 \bigl(X+b(T),\ R(X+b(T),T,Z+w(T)),\
 S(X+b(T),T,Z+w(T))\bigr).
\end{equation}
\end{lemma}
\begin{proof}
After applying $(x,H,W)$, permute the first two coordinates and
subtract the two polynomials $b(T)$ and $w(T)$.  These operations have
the inverse displayed in \eqref{eq:graph-inverse}.
\end{proof}

\subsection{Low-degree coordinates and simple plane projections}

\begin{proposition}\label{prop:quadratic-component}
If an ambient coordinate restricts to a polynomial of degree at most two
on a closed embedding $i:\mathbb A^1_{\mathbb C}\hookrightarrow
\mathbb A^3_{\mathbb C}$, then $i$ is rectifiable.
\end{proposition}
\begin{proof}
Use the coordinate as the first ambient coordinate.  A constant
restriction places the curve in a coordinate plane, so the
Abhyankar--Moh--Suzuki theorem applies \cite{AM,Suzuki}.  A linear
restriction gives a graph.

For a quadratic restriction, translate and scale the source and the
first coordinate so that the curve has first component $t^2$.
Subtracting the even parts of the other components gives
\[
 i(t)=\bigl(t^2,tA(t^2),tB(t^2)\bigr).
\]
The polynomials $A,B\in\mathbb C[x]$ are coprime.  Indeed, a common
nonzero root would identify two distinct parameters $t$ and $-t$;
a common root at zero would make the differential of the embedding
vanish at zero.  Choose $U,V\in\mathbb C[x]$ with $UA+VB=1$.
The matrix
\[
 \begin{pmatrix}U&V\\-B&A\end{pmatrix}
 \in\operatorname{SL}_2(\mathbb C[x])
\]
defines an ambient polynomial automorphism and takes the curve to
$(t^2,t,0)$.  A permutation and the shear $x\mapsto x-y^2$ finish.
\end{proof}

For a parametrization $(f,g,h)$, suppose that
\[
 A=\mathbb C[f(t),g(t)]\subset\mathbb C[t],\qquad
 t-\lambda h(t)\in A\quad(\lambda\ne0),
\]
and choose $P$ with $P(f,g)=t-\lambda h$. The ambient coordinate
$P(x,y)+\lambda z$ restricts to $t$, so the parametrization is
rectifiable.

\begin{corollary}\label{cor:one-defect}
Every closed lift to $\mathbb A^3_{\mathbb C}$ of a finite birational
plane parametrization whose image has exactly one ordinary node, or
exactly one ordinary cusp, and no other singularities, is rectifiable.
\end{corollary}
\begin{proof}
In either case $\dim_{\mathbb C}(\mathbb C[t]/A)=1$.
For a node with preimages $r\ne s$, membership in $A$ is the condition
that the values at $r,s$ agree.  A closed lift has $h(r)\ne h(s)$, so
take $\lambda=(r-s)/(h(r)-h(s))$.
For a cusp with preimage $r$, membership is the condition that the
derivative at $r$ vanish.  Embeddedness gives $h'(r)\ne0$, and
$\lambda=1/h'(r)$ works.  In both cases $t-\lambda h\in A$.
\end{proof}

\subsection{Finite inverses and transfer through a modification}

A polynomial can be invertible modulo $m(x)$ without being invertible
over $k[x]$. Adding a second transverse variable converts such a finite
inverse into an ambient automorphism \cite[Lemma~1.1]{EVM}.

\begin{lemma}\label{lem:finite-inverse}
Let $k$ be an integral domain, let $0\ne m\in k[x]$, and let $P(x,v),G(x,T)$
satisfy
\[
 P(x,G(x,T))\equiv T\pmod m,\qquad
 G(x,P(x,v))\equiv v\pmod m.
\]
Set
\[
 F=P(x,v)+m(x)w,\qquad
 J=\frac{v-G(x,F)}{m(x)}.
\]
Then $J$ is polynomial and $(x,v,w)\mapsto(x,F,J)$ is an automorphism.
Its inverse, on input $(x,T,Z)$, is
\begin{equation}\label{eq:finite-inverse}
 v=G(x,T)+m(x)Z,\qquad
 w=\frac{T-P(x,G(x,T)+m(x)Z)}{m(x)}.
\end{equation}
Both divisions are exact.
\end{lemma}
\begin{proof}
The two congruences prove the divisibility assertions.  Substitution in
the displayed formulas recovers $F=T$, $J=Z$, and, in the other order,
$v$ and $w$. Cancellation of $m$ is legitimate in the polynomial domain.
\end{proof}

The following criterion also applies when the coefficient ring has
nilpotents.

\begin{proposition}\label{prop:one-variable-inverse}
Let $B$ be a nonzero commutative ring. A polynomial
\[
 c(v)=c_0+c_1v+\sum_{j\ge2}c_jv^j\in B[v]
\]
defines a $B$-algebra automorphism of $B[v]$ if and only if $c_1$
is a unit and every $c_j$, $j\ge2$, is nilpotent.
For $B=k[x]/(m)$ with $k$ a field and $m$ nonconstant, this is the
finite test
\[
 \gcd(c_1,m)=1,\qquad c_j\in\sqrt{(m)}\quad(j\ge2),
\]
where coefficient representatives in $k[x]$ are understood. The latter
condition means divisibility by the product of the distinct monic
irreducible factors of $m$.
\end{proposition}
\begin{proof}
Reduction modulo a prime $\mathfrak p$ and passage to
$\operatorname{Frac}(B/\mathfrak p)$ gives a one-variable polynomial
automorphism over a field, hence a polynomial of degree one. Thus every
higher coefficient lies in every prime ideal and is nilpotent. The
linear coefficient lies in no maximal ideal and is a unit.

Conversely the finitely many higher coefficients generate a nilpotent
ideal $N$. Starting with $\psi_0=c_1^{-1}(T-c_0)$, iterate
\begin{equation}\label{eq:nilpotent-inversion}
 \psi_{r+1}=c_1^{-1}\left(T-c_0-\sum_{j\ge2}c_j\psi_r^j\right).
\end{equation}
If $N^e=0$, the differences of successive iterates lie in
$N^{r+1}[T]$, so the iteration stabilizes by the $e$th correction.
It gives $c(\psi)=T$.  Factoring $c(v)-c(w)$ shows that the difference
quotient is a unit, because it is $c_1$ plus a nilpotent polynomial.
Consequently $\psi(c(v))=v$ as well.
\end{proof}

If $m$ is a nonzero constant over a field, then $k[x]/(m)$ is the
zero ring, and the finite-inverse construction applies directly with
division by a unit.

The same construction transfers a coordinate through an affine
modification.

\begin{theorem}\label{thm:coordinate-transfer}
Let $k$ be an integral domain and suppose
\[
 A=(x,H,W),\qquad A^{-1}=(x,R,S)
\]
are polynomial inverse maps over $k$.  Let $0\ne m\in k[x]$ and
$p\in k[x,y]$.  Assume
\[
 c(y)=H(x,p(x,y),y)\in(k[x]/(m))[y]
\]
is a one-variable polynomial automorphism, and let $\psi(T)$ be its
inverse.  Choose a polynomial representative $\varphi(x,T)$ of
\[
 W(x,p(x,\psi(T)),\psi(T))\pmod m.
\]
Then
\begin{equation}\label{eq:transferred-map}
 \left(x,\ H(x,p+mz,y),\
 \frac{W(x,p+mz,y)-\varphi(x,H(x,p+mz,y))}{m}\right)
\end{equation}
is a polynomial automorphism.  Its inverse has transverse components
\begin{equation}\label{eq:transferred-inverse}
 \begin{split}
 y&=S(x,T,mQ+\varphi(x,T)),\\
 z&=\frac{R(x,T,mQ+\varphi(x,T))
 -p(x,S(x,T,mQ+\varphi(x,T)))}{m}.
 \end{split}
\end{equation}
For nonzero $k[x]/(m)$, Proposition~\ref{prop:one-variable-inverse}
shows that the hypothesis on $c$ holds exactly when its linear
coefficient is a unit and each higher coefficient is nilpotent.
\end{theorem}
\begin{proof}
Modulo $m$, the defining identity for $\varphi$ gives
$W(x,p(x,y),y)=\varphi(x,c(y))$.  This proves polynomiality of
\eqref{eq:transferred-map}.  Applying $A^{-1}$ to
$(x,T,\varphi(x,T))$ gives, modulo $m$,
\[
 R(x,T,\varphi)=p(x,\psi(T)),\qquad
 S(x,T,\varphi)=\psi(T).
\]
These congruences prove polynomiality of \eqref{eq:transferred-inverse}.
After inverting $m$, the two maps are inverse by their factorization
through $(x,y,z)\mapsto(x,p+mz,y)$ and
$(x,T,V)\mapsto(x,T,(V-\varphi)/m)$.
Since both complete maps are polynomial and localization of a domain
is injective, both inverse identities already hold over $k$.
\end{proof}

The modification has centre ideal $(m,Y-p(x,Z))$.  The coordinate
change identifies this centre with $(m,V-\varphi(x,T))$.
Thus the construction agrees with the lifting principle for affine
modifications in \cite[Proposition~2.1]{KZ}.

\subsection{Rectifications with a prescribed first component}\label{subsec:all-degree}

For an arbitrary first component, we choose the height by an explicit
congruence and clear only the denominators needed to recover the
parameter $t$.

\begin{theorem}\label{thm:all-degree-family}
Let $b,h\in\mathbb C[t]$, let $n=\deg b\ge2$, and let
$c,\kappa\in\mathbb C^*$.  There is an effective polynomial
$d\in\mathbb C[x]$ with $d(0)=1$ having the following property.
For every $M\ge1$, choose $\mu_0\in\mathbb C[x]$ such that
\[
 x^M\mu_0\equiv\kappa\pmod d.
\]
Then, for every $\rho\in\mathbb C[x]$, the map
\begin{equation}\label{eq:accepted-height-family}
 t\longmapsto\left(b(t),\ ct+b(t)h(t),\
 [\mu_0(b(t))+d(b(t))\rho(b(t))]t\right)
\end{equation}
is a rectifiable closed embedding.  A rectifier and its polynomial
inverse are obtained by finite polynomial operations.
\end{theorem}
\begin{proof}
In the free $\mathbb C[x]$-algebra
\[
 \mathbb C[x,t]/(b(t)-x)
\]
use the basis $1,t,\ldots,t^{n-1}$.  This is a free basis even when
$b$ is not monic, since its leading coefficient is a nonzero complex
constant.  Put $v=ct+xh(t)$, let $K(x)$ have columns the coefficient
vectors of $1,v,\ldots,v^{n-1}$, and set $D=\det K$.
At $x=0$ these columns are $1,ct,\ldots,c^{n-1}t^{n-1}$, so
\[
 D(0)=c^{n(n-1)/2}.
\]
The determinant clears every entry of $K^{-1}$, but we only need
the column which recovers $t$. Let $e_t$ be the coefficient vector
of $t$ and write
\[
 \alpha=K^{-1}e_t\in\mathbb C(x)^n.
\]
Take $d$ to be the least common multiple of the reduced denominators
of the entries of $\alpha$, normalized by $d(0)=1$. Each denominator
divides $D$ up to a nonzero scalar, so $d\mid D$ and $d(0)\ne0$.
The coefficient vector $d\alpha$ gives a polynomial $R_d(x,V)$ of
$V$-degree less than $n$ satisfying
\begin{equation}\label{eq:power-basis-recovery}
 R_d(b(t),ct+b(t)h(t))=d(b(t))t.
\end{equation}
This is exactly the normalized generator of the ideal
\[
 \{s\in\mathbb C[x]:st\in\mathbb C[x][v]\}
\]
inside $\mathbb C[x,t]/(b(t)-x)$. Indeed, the characteristic polynomial
of multiplication by $v$ is monic of degree $n$, so every polynomial
in $v$ reduces to the span of $1,v,\ldots,v^{n-1}$. Their generic
linear independence makes the coefficient vector of $st$ uniquely
$s\alpha$. Its entries belong to $\mathbb C[x]$ precisely when $d$
divides $s$. One may equivalently compute this ideal using Smith
normal form of $K$ over $\mathbb C[x]$.

Put $\mu=\mu_0+d\rho$ and
$p=(\kappa-x^M\mu)/d$. With
\[
 P=\frac{pR_d}{\kappa},\qquad m=\frac{x^M}{\kappa},\qquad F=P+mW,
\]
equation \eqref{eq:power-basis-recovery} gives $F(i(t))=t$.
Moreover $R_d(0,V)=d(0)V/c$, because this identity has degree below $n$
and holds in $\mathbb C[t]/(b)$.  Hence $P(0,V)=V/c$.

Its inverse modulo $x^M$ can be calculated without
factoring $b$.  Start with $G_0=cT$ and for $1\le r<M$ set
\[
 E_r=[x^r]\bigl(P(x,G_{r-1})-T\bigr),\qquad
 G_r=G_{r-1}-cx^rE_r.
\]
Induction gives $P(x,G_{M-1})\equiv T\pmod{x^M}$.
The difference quotient of $P$ is $1/c$ plus a nilpotent polynomial
modulo $x^M$, so the opposite composition is also the identity.
Lemma~\ref{lem:finite-inverse} therefore completes $F$ to a relative
automorphism, with companion
$J=(V-G_{M-1}(x,F))/m$ and the explicit inverse
\eqref{eq:finite-inverse}.  Lemma~\ref{lem:graph-subtraction} finishes.
Since the axis has a polynomial inverse image, this also proves the
closed-embedding assertion in the theorem.
\end{proof}

For $b(t)=t^3$, $h(t)=t^4$, and $c=1$, the relation $t^3=x$ gives
$v=(1+x^2)t$. Thus
\[
 D=(1+x^2)^3,\qquad d=1+x^2,\qquad R_d(x,V)=V.
\]
For the quartic example $b=t^4$, $h=t^2$, $c=\kappa=1$, $M=2$,
the reduced denominator is $d=D=1+4x^5$. The formulas are
\[
 \begin{split}
 d&=1+4x^5,\qquad \mu_0=-4x^3,\\
 R_d(x,V)&=2x^2V^3-xV^2+(1-2x^5)V-5x^4.
 \end{split}
\]
Thus $(t^4,t+t^6,-4t^{13})$ is rectifiable.  Put
$F=R_d+x^2W$ and $J=(V-F-xF^2)/x^2$.
The rectifier is $(F,x-F^4,J)$; here $J$ vanishes on the curve.
Its inverse first puts $x=X+T^4$, and then
\[
 \begin{split}
 V&=S=T+xT^2+x^2Z,\\
 W&=-Z+(T^2+xZ)(S+T)-2S^3+2x^3S+5x^2.
 \end{split}
\]
The identity $x^2W=T-R_d(x,S)$ verifies the inverse directly.

\subsection{Cubic families from polynomial shears}

The following shears apply to a nonmonomial cubic first component.

\begin{proposition}\label{prop:layered-cubic}
Let $b(t)=t^3-3t$.  For arbitrary $r,s,q\in\mathbb C[x]$ and
$\kappa\in\mathbb C^*$ put
\[
 N_r=r^3-3r-x,\qquad c=\kappa+sN_r,\qquad a=r+qc.
\]
Then
\[
 t\longmapsto\bigl(b(t),c(b(t))t,t^2+a(b(t))t\bigr)
\]
is a rectifiable closed embedding.
More generally start with $a_0=r,c_0=\kappa$ and perform finitely many
steps
\[
 c_j=c_{j-1}+s_j(a_{j-1}^3-3a_{j-1}-x),\qquad
 a_j=a_{j-1}+q_jc_j.
\]
Every resulting pair $(a_j,c_j)$ gives a rectifiable embedding of the
same displayed form.
\end{proposition}
\begin{proof}
For one step set
\[
 w=z-q(x)y,\quad
 P_r(x,w)=w^2-(r(x)^2+3)w-2r(x)x,\quad
 T=\frac{y+s(x)P_r(x,w)}\kappa.
\]
These are compositions of polynomial shears and a nonzero constant
scaling.  On the curve, $w=t^2+r(b(t))t$ and
\[
 P_r(b(t),w(t))=-N_r(b(t))t,
\]
which follows by substituting $t^3=3t+b(t)$.  Thus $T=t$.
An explicit final rectifier is
\[
 \bigl(T,\ x-T^3+3T,\ w-T^2-r(x)T\bigr).
\]
It is an automorphism because from its output $(T,X,Z)$ one recovers
$x=X+T^3-3T$, then $w=Z+T^2+r(x)T$, then
$y=\kappa T-sP_r(x,w)$ and $z=w+qy$.
For several steps undo the last shear in $z$, then the last shear in
$y$.  They reduce $(a_j,c_j)$ to $(a_{j-1},c_{j-1})$ and yield the
claim by induction.
\end{proof}

There is another simple subclass.  If $c$ divides $a+k$ for a constant
$k$, then subtracting the corresponding multiple of $y$ from $z$
produces the restriction $t^2-kt$.
Proposition~\ref{prop:quadratic-component} applies whenever the original
parametrization is a closed embedding.

\subsection{Cubic families from the Nagata flow}

We use the Nagata flow
\[
 \begin{split}
 Q&=xZ-Y^2+2,\\
 Y'&=Y+\lambda(x)xQ,\\
 Z'&=Z+2\lambda(x)YQ+\lambda(x)^2xQ^2.
 \end{split}
\]
It preserves $Q$ and has inverse obtained by replacing $\lambda$ by
$-\lambda$; see \cite{EVM}.

\begin{proposition}\label{prop:nagata-family}
Let $c\in\mathbb C^*$ and $r,h,k\in\mathbb C[x]$, and write
$x=b(t)=t^3-3t$ in the following parametrization:
\begin{equation}\label{eq:nagata-family}
 \begin{split}
 Y(t)&=xr(x)t^2+ct+xh(x),\\
 Z(t)&=xr(x)(3r(x)+2h(x))t^2\\
 &\quad+[x^2r(x)^2+2c(3r(x)+h(x))]t+k(x).
 \end{split}
\end{equation}
Then $(b(t),Y(t),Z(t))$ is a rectifiable closed embedding.
\end{proposition}
\begin{proof}
Put
\[
 C_0=xk-2cx^2r-x^2h^2+2,\qquad
 P=xh+\frac{xrC_0}{c^2}.
\]
Using $t^3=3t+x$ gives, on the curve,
$Q=C_0-c^2t^2$.  Apply the displayed flow with $\lambda=r/c^2$.
Then $Y'=ct+P(x)$, and consequently
\[
 H=\frac{Y'-P(x)}c
\]
is an ambient coordinate restricting to $t$. The transformed third
coordinate $Z'$ is a companion to $H$; reversing the flow and the
shear gives the inverse. Graph subtraction completes the rectification.
\end{proof}

For $r=1$, $h=-2$, $k=2x$ the family becomes
\[
 \bigl(C_3(t),\ C_5(t)+(c+1)t,\ C_7(t)+(2c+1)t\bigr),\qquad c\ne0,
\]
where $C_n(t)=2T_n(t/2)$ are the normalized Chebyshev polynomials.
For example $c=-1/2$ gives
$(C_3,C_5+t/2,C_7)$.  The excluded value $c=0$ is not an embedding:
the two parameters $t=\sqrt3$ and $t=-\sqrt3$ have the same image in
\eqref{eq:nagata-family}.

\subsection{Rectification of the Shastri trefoil}
\label{subsec:trefoil}

For the Shastri trefoil \cite{Shastri}, we use the compact integral
formulas of \cite{DorofeevRepository}.

Work over
\[
 \mathcal R=\mathbb Z[a]/(a^2-a+1),\qquad a(1-a)=1.
\]
This is an integral domain, and both $a$ and $a-1$ are units.
Define the polynomial in two variables
\begin{equation}\label{eq:trefoil-h}
 \begin{split}
 h(x,u)={}&(1-a)xu^2+(-1+ax^2-ax^4)u\\
 &-2x^5+(6-2a)x^3-2(1-a)x
 \end{split}
\end{equation}
and its inverse modulo $x^5$:
\begin{equation}\label{eq:trefoil-f}
 \begin{split}
 f(x,T)={}&-T+x\bigl((1-a)T^2+2a-2\bigr)\\
 &+x^2(2aT^3-5aT)\\
 &+x^3(-5T^4+15T^2-2a)\\
 &+x^4\bigl(14(1-a)T^5+50(a-1)T^3+(29-24a)T\bigr).
 \end{split}
\end{equation}
Coefficientwise reduction in $\mathbb Z[a]$ modulo the monic polynomial
$a^2-a+1$ gives
\begin{equation}\label{eq:trefoil-finite-identities}
 h(x,f(x,T))\equiv T\pmod{x^5},\qquad
 f(x,h(x,u))\equiv u\pmod{x^5}.
\end{equation}
Set
\begin{equation}\label{eq:trefoil-base-coordinates}
 \begin{split}
 u&=z+axy,\\
 H(x,y,z)&=h(x,u)+(a-1)x^5y,\\
 W(x,y,z)&=\frac{a(u-f(x,H(x,y,z)))}{x^5}.
 \end{split}
\end{equation}
The last quotient is exact in $\mathcal R[x,y,z]$ by
\eqref{eq:trefoil-finite-identities}. The inverse of $A_0=(x,H,W)$
is given on $(x,T,V)$ by
\begin{equation}\label{eq:trefoil-relative-inverse}
 \begin{split}
 U&=f(x,T)+(1-a)x^5V,\\
 R(x,T,V)&=\frac{a(h(x,U)-T)}{x^5},\\
 S(x,T,V)&=U-axR(x,T,V).
 \end{split}
\end{equation}
The quotient is exact in $\mathcal R[x,T,V]$. These are the inverse
formulas in Lemma~\ref{lem:finite-inverse}, after the linear change
$(y,z)\mapsto(u,y)$, with $m=(a-1)x^5$ and companion $-W$.
Thus $(x,R,S)=A_0^{-1}$ over $\mathcal R$, on all fibres including
$x=0$. Divisibility by $x^5$ is coefficientwise in the free
$\mathbb Z$-module $\mathcal R=\mathbb Z\oplus\mathbb Z a$;
no integer is inverted.

The normalized Chebyshev polynomials needed here are
\[
 \begin{split}
 C_3(t)&=t^3-3t,\qquad C_4(t)=t^4-4t^2+2,\\
 C_5(t)&=t^5-5t^3+5t,\\
 C_7(t)&=t^7-7t^5+14t^3-7t.
 \end{split}
\]
Direct substitution in the compact polynomials gives
\begin{equation}\label{eq:trefoil-marking}
 H(C_3,C_4,C_5)=t,\qquad
 W(C_3,C_4,C_5)=8-14t^2.
\end{equation}
The affine change
\[
 M(x,y,z)=(x,y+2,z-5x)
\]
takes Shastri's parametrization
\[
 \gamma_{\mathrm{trefoil}}(t)=
 (t^3-3t,t^4-4t^2,t^5-10t)
\]
to $(C_3,C_4,C_5)$.  This normalization is also recorded in
\cite{FF}.  Write $h_{\mathrm{trefoil}}=H\circ M$ and
$w_{\mathrm{trefoil}}=W\circ M$.

\begin{samepage}
\begin{theorem}\label{thm:trefoil-rectification}
The map
\begin{equation}\label{eq:trefoil-rectifier}
 \bigl(h_{\mathrm{trefoil}},\
 x-h_{\mathrm{trefoil}}^3+3h_{\mathrm{trefoil}},\
 w_{\mathrm{trefoil}}-8+14h_{\mathrm{trefoil}}^2\bigr)
\end{equation}
is a polynomial automorphism over $\mathcal R$ and rectifies
$\gamma_{\mathrm{trefoil}}$.  Its inverse, on $(T,X,Z)$, is
\begin{equation}\label{eq:trefoil-final-inverse}
 \begin{split}
 x&=X+T^3-3T,\\
 y&=R(x,T,Z+8-14T^2)-2,\\
 z&=S(x,T,Z+8-14T^2)+5x.
 \end{split}
\end{equation}
After extension to characteristic zero, the component degrees of the
displayed rectifier are $(5,15,24)$.
\end{theorem}
\end{samepage}
\begin{proof}
Apply Lemma~\ref{lem:graph-subtraction} to
\eqref{eq:trefoil-marking} and then undo $M$.  The inverse formulas
are \eqref{eq:trefoil-relative-inverse} followed by $M^{-1}$.
Expansion gives the stated degrees.
\end{proof}

\subsection{Rectification of the harmonic figure-eight}

The harmonic figure-eight is
\[
 \gamma_{\mathrm{harmonic}}(t)=(C_3(t),C_5(t),C_7(t)).
\]
Its real knot identification is discussed in \cite{KP}. The relevant identity is
\begin{equation}\label{eq:chebyshev-modification}
 C_3C_7-C_5^2+2=C_4.
\end{equation}
It explains why the rationally invertible modification
$\beta(x,y,z)=(x,xz-y^2+2,y)$ links this curve to the already treated
normalized trefoil.

Since $H(0,Y,Z)=-Z$, the centre
$H(0,2-y^2,y)=-y$ satisfies
Theorem~\ref{thm:coordinate-transfer}.  The correction polynomial is
\begin{equation}\label{eq:harmonic-correction}
 \begin{split}
 \varphi(T)&=W(0,2-T^2,-T)\\
 &=42(a-1)T^6+175(1-a)T^4+(147a-161)T^2+8.
 \end{split}
\end{equation}
Define
\begin{equation}\label{eq:harmonic-relative}
 \begin{split}
 H_{\mathrm{harmonic}}&=H(x,xz-y^2+2,y),\\
 Q_{\mathrm{harmonic}}&=
 \frac{W(x,xz-y^2+2,y)-\varphi(H_{\mathrm{harmonic}})}x.
 \end{split}
\end{equation}
Both are polynomial, and the relative inverse has components
\begin{equation}\label{eq:harmonic-relative-inverse}
 \begin{split}
 R_{\mathrm{harmonic}}(x,T,Q)&=S(x,T,xQ+\varphi(T)),\\
 S_{\mathrm{harmonic}}(x,T,Q)&=
 \frac{R(x,T,xQ+\varphi(T))
 +S(x,T,xQ+\varphi(T))^2-2}{x}.
 \end{split}
\end{equation}
The transfer theorem over $\mathcal R$, with $m=x$, proves exact
division in both directions over that ring.
The identity \eqref{eq:chebyshev-modification} and substitution give
\[
 H_{\mathrm{harmonic}}(\gamma_{\mathrm{harmonic}}(t))=t,\qquad
 Q_{\mathrm{harmonic}}(\gamma_{\mathrm{harmonic}}(t))
 =(1-a)(42t^3-49t).
\]

\begin{theorem}\label{thm:harmonic-rectification}
The harmonic figure-eight admits the following polynomial rectifier
over $\mathcal R$:
\begin{equation}\label{eq:harmonic-final}
 \begin{split}
 \bigl(&H_{\mathrm{harmonic}},\
 x-H_{\mathrm{harmonic}}^3+3H_{\mathrm{harmonic}},\\
 &Q_{\mathrm{harmonic}}
 -(1-a)(42H_{\mathrm{harmonic}}^3-49H_{\mathrm{harmonic}})\bigr).
 \end{split}
\end{equation}
Its inverse is obtained by putting $x=X+T^3-3T$ and evaluating
\eqref{eq:harmonic-relative-inverse} at
$Q=Z+(1-a)(42T^3-49T)$. After extension to characteristic zero,
its component degrees are $(7,21,41)$.
\end{theorem}
\begin{proof}
Apply graph subtraction to \eqref{eq:harmonic-relative} and expand
in characteristic zero to obtain the stated degrees.
\end{proof}

The second component here contains $+5t$; the variant with $+4t$ is
a different parametrization and is not covered by this construction.
Every identity just proved survives every ring homomorphism
$\mathcal R\to B$. In particular, both displayed parametrizations
are rectifiable over every field containing a root of $a^2-a+1$.
This includes characteristic two and characteristic three; in the
latter, the root is $a=-1$. Every finite field has an extension of
degree at most two containing such a root. The specialized
parametrizations remain closed embeddings. Their degrees may drop
under reduction.

Embedding $\mathcal R$ in $\mathbb C$ recovers the complex
rectifications over $\mathbb Q(\sqrt{-3})$. This is compatible with
the real knot obstruction: $a^2-a+1$ has no real root, and no real
polynomial rectifier is asserted.

Finally the construction admits arbitrarily high exceptional powers.
For $N\ge1$ define
\[
 \tau_0(x,y)=-y,\qquad
 \tau_{j+1}(x,y)=-y+x(\tau_j(x,y)^2-2),
\]
and let $p_N(x,y)$ be the remainder modulo $x^N$ of
$x\tau_{N-1}-\tau_{N-1}^2+2$.

\begin{corollary}\label{cor:all-power-transfer}
For every $N\ge1$ the expression
\[
 e_N(t)=\frac{C_4(t)-p_N(C_3(t),C_5(t))}{C_3(t)^N}
\]
belongs to $\mathbb Z[t]$, and $(C_3,C_5,e_N)$ is rectifiable over
$\mathcal R$ and after every base change from $\mathcal R$.
\end{corollary}
\begin{proof}
The identity $C_5=-t+C_3(t)(t^2-2)$ gives inductively
$\tau_{N-1}(C_3,C_5)\equiv t\pmod{C_3^N}$.
Also $C_4=C_3t-t^2+2$.  These identities prove the exact division.
All recursions have integer coefficients; division by the monic
polynomial $C_3^N$ is therefore exact already in $\mathbb Z[t]$.
The modification $(x,y,z)\mapsto(x,p_N+x^Nz,y)$ takes this curve to
$(C_3,C_4,C_5)$. Modulo $x$ its centre under $H$ is $-y$.
In $\mathcal R[x]/(x^N)$, its linear coefficient is thus a unit and
its higher coefficients lie in the nilpotent ideal $(x)$.
Proposition~\ref{prop:one-variable-inverse} and the transfer theorem
supply an actual coordinate restricting to $t$, a companion, and a
polynomial inverse over $\mathcal R$. Graph subtraction and base change
finish the proof.
\end{proof}

\section{Normal maps and rectifiability}
\label{sec:geometric-realization}

We now consider how the framed normal map might lead to polynomial
coordinates, as asked in Question~\ref{qu:intro-main}. A polynomial
section and an affine-line generic fiber suffice. We first prove this
criterion and give an algorithm for the inverse, then explain what
remains to be constructed from the motivic comparison.

\subsection{Normal maps and coordinate projections}

Let $i:\mathbb A^1_{\mathbb C}\hookrightarrow\mathbb A^3_{\mathbb C}$
be a closed embedding. An ordered pair of generators
$I_{i(\mathbb A^1)}=(f,g)$ determines both a normal framing along the
curve and a polynomial map
\[
 (f,g):\mathbb A^3_{\mathbb C}\longrightarrow\mathbb A^2_{\mathbb C}.
\]
Its restriction to the complement takes values in the punctured plane.
An $\mathbb A^1$-equivalence between the complements need not make this
map an affine-line fibration: absolute motivic localization does not
preserve its fibers.

We seek ordered generators for which the normal map has a polynomial
section and a geometric generic fiber isomorphic to the affine line.
The criterion below then makes the entire map a coordinate projection,
including its exceptional fibers.

The section must satisfy $p\circ e=\operatorname{id}$ as an equality
of morphisms, so that it induces a retraction of polynomial rings;
a homotopy does not suffice. The argument works over fields that need
not be algebraically closed.

\subsection{The principal-retraction criterion}

The algebraic input is the principal-kernel observation of Das and
Dutta \cite[Lemma~3.6 and Remark~3.7]{DasDutta2011}, together with
descent from a geometric generic line.

\begin{lemma}\label{lem:principal-retraction}
Let $B\subset A$ be integral domains, let $K=\operatorname{Frac}(B)$,
and let $\epsilon:A\to B$ be a $B$-algebra retraction. Suppose that
$\ker(\epsilon)=hA\ne 0$ and
\[
 A\otimes_B\overline K\cong\overline K[w]
\]
for an algebraic closure $\overline K$ of $K$. Then the given inclusion
$B[h]\subset A$ is an equality, and $h$ is algebraically independent
over $B$.
\end{lemma}

\begin{proof}
Localizing and extending scalars preserves the exact sequence defining
the kernel. A $\overline K$-algebra retraction of
$\overline K[w]$ onto $\overline K$ is evaluation at a scalar
$a\in\overline K$. Its kernel is generated by $w-a$. Consequently the
image of $h$ is $c(w-a)$ for some $c\in\overline K^{\times}$. The
natural map
\[
 K[H]\longrightarrow A\otimes_B K,\qquad H\longmapsto h,
\]
is therefore an isomorphism after the faithfully flat extension
$K\subset\overline K$, and hence is an isomorphism already over $K$.

For $r\in A$, set $r_0=r$ and successively define
\[
 b_j=\epsilon(r_j),\qquad
 r_{j+1}=\frac{r_j-b_j}{h}.
\]
Every quotient belongs to $A$, since the numerator lies in the
principal kernel. In $A\otimes_B K=K[h]$, the retraction is evaluation
at $h=0$. Thus each nonconstant remainder has strictly smaller degree
in $h$. If a remainder has degree zero, it equals its retraction and
belongs to $B$, so the next remainder is zero. The process terminates,
giving
\[
 r=b_0+b_1h+\cdots+b_dh^d,\qquad b_j\in B.
\]
This proves $A=B[h]$. The same identification with $K[h]$ shows that
there is no polynomial relation for $h$ over $B$.
\end{proof}

\begin{theorem}\label{thm:section-generic-line}
Let $k$ be any field, and let
$p=(f_1,\ldots,f_{n-1}):\mathbb A^n_k\to\mathbb A^{n-1}_k$
be a polynomial map. If $p$ has a polynomial section and its geometric
generic fiber is $\mathbb A^1$, then there is a polynomial $h$ such
that
\[
 k[x_1,\ldots,x_n]=k[f_1,\ldots,f_{n-1},h].
\]
In particular, $(f_1,\ldots,f_{n-1},h)$ is a polynomial
automorphism.
\end{theorem}

\begin{proof}
The section makes the map
$k[t_1,\ldots,t_{n-1}]\to k[x_1,\ldots,x_n]$, $t_j\mapsto f_j$,
injective. Put $A=k[x_1,\ldots,x_n]$ and
$B=k[f_1,\ldots,f_{n-1}]$. Substitution along the section defines a
$B$-algebra retraction $\epsilon:A\to B$. Its kernel is a prime ideal
of height one: the quotient is the polynomial ring $B$ of dimension
$n-1$. Since $A$ is factorial, this kernel equals $hA$ for a
nonzero polynomial $h$. Lemma~\ref{lem:principal-retraction} applies
and gives the asserted equality.

Applying the terminating division procedure to each $x_j$ expresses
it as a polynomial in $f_1,\ldots,f_{n-1},h$. These expressions are
the coordinates of the inverse morphism. Algebraic independence gives
both inverse identities.
\end{proof}

The proof gives an algorithm for the inverse when the section
$e=(e_1,\ldots,e_n)$ in Theorem~\ref{thm:section-generic-line} is known.

\begin{enumerate}
\item Form the retraction by substitution:
\[
 \epsilon(r)=r\bigl(e_1(f_1,\ldots,f_{n-1}),\ldots,
                         e_n(f_1,\ldots,f_{n-1})\bigr).
\]
Put $g_j=x_j-\epsilon(x_j)$ for $1\le j\le n$. Then
\begin{equation}\label{eq:section-kernel-generators}
 \ker\epsilon=(g_1,\ldots,g_n).
\end{equation}
Indeed, idempotence of $\epsilon$ kills every $g_j$.
For any polynomial $r$, the difference $r(x)-r(\epsilon(x))$
belongs to this ideal: for a monomial, change its variables one at
a time and factor each difference of powers. Taking linear
combinations proves the assertion for every $r$. If
$\epsilon(r)=0$, this difference is $r$ itself, proving the reverse
inclusion in \eqref{eq:section-kernel-generators}.

\item Compute a greatest common divisor $h$ of the nonzero $g_j$ in
the polynomial unique factorization domain $k[x_1,\ldots,x_n]$.
At least one $g_j$ is nonzero, since otherwise $\epsilon$ would be
the identity and its image would have dimension $n$, rather than
$n-1$. The theorem proves that the ideal in
\eqref{eq:section-kernel-generators} is principal. Therefore its
generator is precisely this greatest common divisor, up to a
nonzero scalar; any scalar normalization may be used. The
cofactors $g_j/h$ are polynomial. Polynomial ideal-membership
computation also gives
$u_1,\ldots,u_n$ with $h=\sum_j u_jg_j$.

\item For each ambient coordinate $x_j$, start with $r_{j,0}=x_j$
and record
\begin{equation}\label{eq:section-inverse-algorithm}
 \begin{split}
 b_{j,\nu}(T)&=
 r_{j,\nu}\bigl(e_1(T),\ldots,e_n(T)\bigr),\\
 r_{j,\nu+1}&=
 \frac{r_{j,\nu}-b_{j,\nu}(f_1,\ldots,f_{n-1})}{h}.
 \end{split}
\end{equation}
Stop when the remainder is zero. Every division is exact because
its numerator lies in $\ker\epsilon=hA$. Termination is guaranteed
by the decrease of degree in the generic parameter $h$ proved in
Lemma~\ref{lem:principal-retraction}. In particular it requires no
prior bound on the total degrees of these ambient remainders.
\end{enumerate}

If $d_j$ is the last index with a nonzero recorded coefficient, put
\[
 R_j(T_1,\ldots,T_{n-1},H)
       =\sum_{\nu=0}^{d_j}b_{j,\nu}(T_1,\ldots,T_{n-1})H^\nu.
\]
The polynomials $(R_1,\ldots,R_n)$ give the inverse of
$(f_1,\ldots,f_{n-1},h)$. The exact divisions give
$R_j(f_1,\ldots,f_{n-1},h)=x_j$. The converse identities
$f_i(R_1,\ldots,R_n)=T_i$ and $h(R_1,\ldots,R_n)=H$ follow from
the algebraic independence in the theorem and can also be checked
directly as polynomial identities. If arithmetic in $k$ is effective,
multivariate gcd and exact division compute this inverse from the section.

For example, the map
\[
 (x,y,z)\longmapsto(x+z^2,\ y+xz)
\]
has section $(T_1,T_2)\mapsto(T_1,T_2,0)$. The kernel generators
above are $-z^2,-xz,z$, whose greatest common divisor is $h=z$.
The inverse is
\[
 (T_1-H^2,\ T_2-(T_1-H^2)H,\ H).
\]
\begin{corollary}\label{cor:normal-projection-rectification}
Let $i:\mathbb A^1_k\hookrightarrow\mathbb A^3_k$ be a closed
embedding, and suppose that $I_{i(\mathbb A^1)}=(f,g)$. If the
polynomial map $(f,g):\mathbb A^3_k\to\mathbb A^2_k$ has a
polynomial section and geometric generic fiber $\mathbb A^1$, then
$i$ is rectifiable over $k$.
\end{corollary}

\begin{proof}
The theorem gives $k[x,y,z]=k[f,g,h]$. Taking the quotient by
$(f,g)$ identifies the coordinate ring of the embedded line with
$k[h]$. Therefore $h(i(t))=at+b$ for some $a\in k^{\times}$ and
$b\in k$. The polynomial automorphism
\[
 \left(\frac{h-b}{a},f,g\right)
\]
sends $i(t)$ to $(t,0,0)$.
\end{proof}

\subsection{Dependence on the choice of generators}

It would be too strong to ask that every ordered generating pair
satisfy the hypotheses of the corollary. The following example retains
the exact zero scheme, a polynomial section, and a polynomial homotopy
inverse on the complement, while its generic fiber is a punctured line.

\begin{example}\label{ex:normal-map-punctured-generic-fiber}
Over any field $k$, consider the coordinate line
$V(x,y)\subset\mathbb A^3_k$. For integers $N,d\geq 1$, put
\[
 p_{N,d}(x,y,z)=\bigl(x,y(1+x^Nz^d)\bigr),\qquad
 e(u,v)=(u,v,0).
\]
The two components generate exactly $(x,y)$, because
\[
 y=y(1+x^Nz^d)-yx^{N-1}z^d x.
\]
The map $e$ is a polynomial section. Moreover,
\[
 H_s(x,y,z)=\bigl(x,y(1+s x^Nz^d),(1-s)z\bigr)
\]
is a polynomial homotopy from the identity to $e\circ p_{N,d}$.
The first two components generate $(x,y)$ over $k[s,x,y,z]$, so this
homotopy restricts to the complement. Thus the restriction of
$p_{N,d}$ to the complement is a polynomial homotopy
equivalence with the punctured plane.

For $K=k(u,v)$, however, its generic fiber has coordinate ring
\[
 K\bigl[z,(1+u^Nz^d)^{-1}\bigr].
\]
The nonconstant element $1+u^Nz^d$ is a unit, also after extension to
$\overline K$. This excludes a geometric generic affine line. The
homotopy $H_s$ is an absolute homotopy; it does not preserve
$p_{N,d}$ fiberwise.

The normal map agrees with the coordinate projection $(x,y)$ modulo
$(x,y)^{N+1}$. Therefore agreement to any specified finite order along
the curve does not fix this difficulty. For this example the generator
change displayed above repairs the normal pair to $(x,y)$, so the
center is of course rectifiable. A change of ideal generators with
ambient polynomial coefficients should not itself be called an ambient
polynomial automorphism.
\end{example}

Thus the criterion must allow a choice of generators; the example
does not decide Question~\ref{qu:intro-main}.

\subsection{The motivic realization problem}

In the classical longitude argument, Dehn's lemma turns a
null-homotopy into an embedded disk. To use that analogy here, one
would need to construct the polynomial section and the affine-line
generic fiber of Corollary~\ref{cor:normal-projection-rectification}
from motivic data. We do not know such a theorem.

The meridian itself cannot be the class filled inside the
complement. Even for the coordinate line, its normal meridian
$\A^2_{\C}\setminus\{0\}$ has nonzero first
$\A^1$-homology, so it cannot extend to a morphism
$\A^2_{\C}\to\A^3_{\C}\setminus i(\A^1_{\C})$.
Such an extension would factor its homology map through the
contractible affine plane. A transverse normal-plane map instead
lands in the ambient $\A^3_{\C}$ and meets the curve.

An algebraic version of the disk argument would face a further
coordinate problem. A finite birational map from an affine plane can
have a nonzero normalization defect, which a framing on its source
does not eliminate. For a closed affine-plane disk meeting the line
in one reduced point, one would still need to recognize its equation
as a coordinate. This is a separate Abhyankar--Sathaye problem.

The existence of generators for which the normal map has a polynomial
section and geometric generic fiber $\mathbb A^1$ is an equivalent
formulation of Question~\ref{qu:intro-main}. One direction is
Corollary~\ref{cor:normal-projection-rectification}; conversely, the
coordinates of a rectifier give such generators and a section.

The example above also shows why a formal inverse need not be
polynomial. The expansion
\[
 y=v\sum_{j\geq0}(-u^N)^jz^{dj}
\]
exists to every finite order but never terminates. In contrast, the
division in Lemma~\ref{lem:principal-retraction} terminates because the
generic fiber is a polynomial ring with the specified parameter.

There is another equivalent formulation using degrees of inverse maps
along the contraction to the tangent line. Translate the embedding so
that $i(0)=0$ and form
\[
 i_s(t)=s^{-1}i(st).
\]
This expression is polynomial in $(s,t)$, with its value at $s=0$
understood by expansion of the component polynomials.

\begin{proposition}\label{prop:bounded-inverse-jets}
Suppose there is an integer $D$ such that, for every $r\geq1$, there
are polynomial inverse automorphisms $\Phi_r,\Psi_r$ of
$\mathbb A^3_{\mathbb C[s]/(s^r)}$, each of total degree at most $D$
in the three spatial variables, and
\[
 \Phi_r(i_s(t))=(t,0,0)\pmod{s^r}.
\]
Then $i$ is rectifiable over $\mathbb C$. Compatibility between the
choices for different $r$ is not required. Conversely, if $i$ is
rectifiable, these pairs can be chosen as the reductions of one
polynomial inverse pair over $\mathbb C[s]$, with a common bound
on their spatial degrees.
\end{proposition}

\begin{proof}
Introduce coefficients for two polynomial maps in three variables of
degree at most $D$, together with a variable $s$. The equations that
the two composites be the identity and that the first map send
$i_s(t)$ to $(t,0,0)$ are finitely many polynomial equations in
these coefficients and $s$. They define an affine scheme $Z$ of
finite type over $\mathbb C[s]$.

The assumed inverse pairs give points of $Z$ over
$\mathbb C[s]/(s^r)$ with the indicated value of the base parameter.
If $Z$ had no complex point over $s\ne0$, the Nullstellensatz would
make $s$ nilpotent in its coordinate ring, say $s^M=0$. A point over
$\mathbb C[s]/(s^r)$ for $r>M$ contradicts this equality. Thus there
is a complex point over some $s=a\ne0$. Its inverse pair gives a
polynomial automorphism $\Phi$ rectifying $i_a$. The automorphism
$x\mapsto a\Phi(x/a)$ then rectifies $i$.

Conversely, let $\Phi$ be a rectifier with
$\Phi(i(t))=(t,0,0)$. Since $i(0)=0$, both $\Phi$ and its inverse
fix the origin. Therefore
\[
 \Phi_s(x)=s^{-1}\Phi(sx),\qquad
 \Psi_s(x)=s^{-1}\Phi^{-1}(sx)
\]
are polynomial maps over $\mathbb C[s]$: a monomial of spatial
degree $d\geq1$ acquires the factor $s^{d-1}$. Their composition
identities hold after inverting $s$ and hence over $\mathbb C[s]$.
Moreover,
\[
 \Phi_s(i_s(t))=s^{-1}\Phi(i(st))=(t,0,0).
\]
Their spatial degrees are bounded by
$\max(\deg\Phi,\deg\Phi^{-1})$. Reduction modulo $s^r$ gives
the required pairs for every $r$.
\end{proof}

For Question~\ref{qu:intro-main}, the corresponding quantitative
problem is to bound $D$ in terms of the degree of the embedding.
For algebraic input coefficients, a known bound reduces the search for
a rectifier to finitely many coefficient equations over an algebraic
closure.

A comparison over a selected coordinate base must also account for
its fibers. If an ambient coordinate
restricts to a polynomial of degree $m>1$ on the embedded line, a
regular fiber of that coordinate meets the curve in $m$ points. Its
complement is an affine plane with $m$ punctures. A comparison over
this base must retain all $m$ meridians; it cannot replace this fiber by
a one-puncture plane. Under complex realization, the map induced by
any exact normal pair on third homology is
\[
 \mathbb Z^m\longrightarrow\mathbb Z,\qquad
 (n_1,\ldots,n_m)\longmapsto n_1+\cdots+n_m.
\]
Each local degree is $+1$ because the intersections are transverse and
complex oriented. Thus an absolute motivic equivalence need not be an
equivalence over the selected coordinate base. Such a comparison must
keep the finite morphism
$i(\mathbb A^1_{\mathbb C})\to\mathbb A^1_{\mathbb C}$ induced by
that coordinate, together with its ramification and nonreduced fibers.

There is also a descent question for the normal comparison over a
function field.

\begin{question}\label{question:generic-field-descent}
Let $k=\mathbb C(s)$, let
$i:\mathbb A^1_k\hookrightarrow\mathbb A^3_k$ be a closed embedding,
and let
\[
 f:\mathbb A^3_k\setminus i(\mathbb A^1_k)
        \longrightarrow\mathbb A^2_k\setminus\{0\}
\]
be a fixed framed normal map with compatible $k$-rational basepoints.
Suppose a pointed motivic meridian class
\[
 m:\mathbb A^2_k\setminus\{0\}
        \longrightarrow\mathbb A^3_k\setminus i(\mathbb A^1_k)
\]
is defined over $k$ and satisfies $f\circ m=\operatorname{id}$ in
the pointed unstable $\mathbb A^1$-homotopy category. If
$(m\circ f)_L=\operatorname{id}$ after a finite separable extension
$L/k$, must $m\circ f=\operatorname{id}$ already over $k$?
\end{question}

After passage to a Galois closure, uniqueness makes the inverse class
Galois invariant. This does not establish the remaining identity over
$k$, even when the candidate inverse $m$ is defined over $k$. A nontrivial
finite separable extension is not a Nisnevich cover of
$\operatorname{Spec}k$, so this does not follow from Nisnevich descent
\cite{MV99}.

\appendix
\section{Other ground fields and arithmetic descent}\label{appendix:other-fields}

Polynomial inverse identities specialize over their coefficient ring.
The motivic computation also depends on field hypotheses and normal
orientations.

\subsection{Homotopy purity over arbitrary fields}

Let $k$ be a field and let
$i:\mathbb A^1_k\hookrightarrow\mathbb A^3_k$ be a closed immersion.  Its normal
bundle is a rank-two vector bundle on $\mathbb A^1_k$ and is therefore trivial.
After choosing a normal frame, homotopy purity gives
\[
 \mathbb A^3_k/
 \bigl(\mathbb A^3_k\setminus i(\mathbb A^1_k)\bigr)
 \simeq_{\mathbb A^1} S^{4,2}.
\]
Here $S^{4,2}=(\mathbb P^1)^{\wedge2}$ denotes the pointed motivic sphere.
This identifies the pointed quotient, rather than the unsuspended complement.

There is also a global defining pair.  Indeed, the codimension-two Serre
construction applies to the smooth curve, whose determinant normal bundle
is trivial.  It supplies a rank-two projective module on $k[x,y,z]$ and a
section with this zero scheme.  Quillen--Suslin makes the module free.
Consequently the ideal of the curve has generators $f,g$ forming a regular
sequence; see also \cite[Theorem~2.6]{WW26}.  Naturality of purity for the transverse square defined by
$(f,g)$ identifies the corresponding quotient map with the Thom map induced
by $i(\mathbb A^1_k)\to\operatorname{Spec}k$ and the selected frame.

The homotopy and homology sheaves below are Nisnevich sheaves on
smooth $k$-schemes; $\mathbf K^{\mathrm{MW}}_2$ denotes the unramified
Milnor--Witt sheaf over $k$.

\begin{proposition}\label{prop:other-field-homology}
Suppose that $k$ is perfect of characteristic different from two.  For an
ordered defining pair $f,g$ of a closed embedded affine line, the morphism
\[
 (f,g):\mathbb A^3_k\setminus i(\mathbb A^1_k)
       \longrightarrow\mathbb A^2_k\setminus\{0\}
\]
induces an isomorphism on first $\mathbb A^1$-homology sheaves.  In particular,
\[
 \mathbf H^{\mathbb A^1}_1
 \bigl(\mathbb A^3_k\setminus i(\mathbb A^1_k)\bigr)
 \cong\mathbf K^{\mathrm{MW}}_2.
\]
If $k$ is also infinite, the complement is $\mathbb A^1$-connected.  If its
first $\mathbb A^1$-homotopy sheaf is abelian, the same defining pair induces
\[
 \pi^{\mathbb A^1}_1
 \bigl(\mathbb A^3_k\setminus i(\mathbb A^1_k)\bigr)
 \xrightarrow{\ \sim\ }\mathbf K^{\mathrm{MW}}_2.
\]
\end{proposition}

\begin{proof}
Apply reduced $\mathbb A^1$-homology to the purity comparison.  The ambient
affine spaces are contractible, and the quotient comparison is an
equivalence.  The long exact homology sequences therefore identify the
first homology of the displayed complements.  Morel's punctured-affine-space
calculation identifies the latter with $\mathbf K^{\mathrm{MW}}_2$.

For connectivity, work over any finitely generated separable extension
$L/k$.  Given two points of the complement, the forbidden intermediate
points for an affine line meeting the curve lie in the closures of two
cones, each of dimension at most two.  An infinite field supplies an
$L$-point outside their union.  Two affine-line maps then join the given
points.  The chain-connectivity criterion of Asok--Doran gives
$\mathbb A^1$-connectivity.

Finally an abelian strongly $\mathbb A^1$-invariant sheaf is strictly
$\mathbb A^1$-invariant over a perfect field.  In this case the degree-one
Hurewicz map is an isomorphism.  Its naturality, applied to the first part,
proves the last assertion \cite{Morel12,AD09}.
\end{proof}

\subsection{Orientations and values of sheaves}

Milnor--Witt theory retains the quadratic information in a normal
orientation.  Recall that $\langle a\rangle$ denotes the one-dimensional
symmetric form with coefficient $a\in k^\times$, regarded as an element of
$\mathrm{GW}(k)=K^{\mathrm{MW}}_0(k)$.  Changing a normal frame by determinant
$a$ multiplies the associated Milnor--Witt Thom identification by
$\langle a\rangle$; equivalently one can keep the determinant-line twist
instead of choosing an orientation.  For the normal bundle on an affine
line, the determinant of a change of frame is a unit in $k[t]$, hence is
constant \cite[Corollary~3.23]{Morel12}. Over $\mathbb C$ it is a square,
so its action is the identity.

The sheaf $\mathbf K^{\mathrm{MW}}_2$ retains values on smooth schemes
and their function fields, beyond the single group $K^{\mathrm{MW}}_2(k)$.
Thus vanishing at a finite ground field does not imply vanishing of the
sheaf. Likewise the comparison with Milnor $K$-theory over $\mathbb C$
does not identify the two sheaves on all smooth complex schemes.

For example, the Milnor symbol
$\{t,-1\}\in K^{\mathrm M}_2(\mathbb C(t))$ has tame residue $-1$ at
$t=0$, so it is nonzero.  Thus the vanishing of constant symbols
$\{a,-1\}$ for $a\in\mathbb C^\times$ cannot be used for every unit of
this function field, even though it contains a square root of $-1$.

\subsection{Real knots and complex rectifications}

The real and complex normal meridians are a circle and a three-sphere,
respectively.

For example, Shastri's polynomial trefoil is
\[
 t\longmapsto(t^3-3t,\ t^4-4t^2,\ t^5-10t).
\]
Its real locus is a nontrivial long knot \cite{Shastri}.  A real
polynomial rectifier would be a proper homeomorphism of $\mathbb R^3$ and
would identify its complement with the complement of a straight line.
The trefoil complement has nonabelian knot group, while the latter
complement has fundamental group $\mathbb Z$.  Such a real rectifier is
therefore impossible, although a complex rectifier exists.

\subsection{Coefficient rings and specialization}

If polynomial tuples satisfy
\[
 \Phi\circ\Psi=\Psi\circ\Phi=\operatorname{id},\qquad
 \Phi\circ i(t)=(t,0,0)
\]
over a commutative coefficient ring $R$, the identities remain valid
after every homomorphism $R\to R'$.  Exact quotients used in a formula must
first be constructed as polynomials over $R$; division by a coefficient
which becomes zero after specialization has no such interpretation.
Thus a rectifier and its inverse over $\mathbb Z[1/N]$ specialize in
characteristics not dividing $N$. For formulas over a number field,
any field embedding preserves the identities. A rectifier
obtained only modulo a prime does not by itself give one in characteristic
zero.

The integral rectifiers of Theorem~\ref{thm:intro-concrete} are defined over
\[
 \mathcal R=\mathbb Z[a]/(a^2-a+1).
\]
Theorems~\ref{thm:trefoil-rectification} and
\ref{thm:harmonic-rectification} therefore specialize over every field
containing a root of $a^2-a+1$. In characteristic three this root is
$a=-1$; in characteristic two the roots lie in $\mathbb F_4$.
Over every finite field an extension of degree at most two suffices.
The inverse identities also survive nonreduced base changes.
The printed degree tuples are characteristic-zero degrees and may
decrease under reduction.

Fixing degree bounds for a polynomial map and its inverse gives an affine
parameter scheme of finite type.

\begin{proposition}\label{prop:bounded-rectifier-scheme}
Fix a polynomial closed immersion $i$ over a field $k$ and integers $D,E$.
The pairs of polynomial tuples $(\Phi,\Psi)$ with degrees at most $D,E$
satisfying the preceding three identities form an affine scheme of finite
type over $k$.  If this scheme has a point over a field extension of $k$,
it has a point over a finite algebraic extension of $k$.
\end{proposition}

\begin{proof}
Introduce one coefficient variable for every monomial allowed in each
tuple.  Equating coefficients in the compositions and in the restriction
to $i$ gives finitely many polynomial equations.  An extension-field point
makes the resulting coordinate algebra nonzero.  A maximal ideal of this
finite-type algebra has residue field finite over $k$, by Zariski's lemma.
\end{proof}

For a fixed model $i$ over $\mathbb Z[1/N]$ and fixed degree bounds
$D,E$, the same coefficient equations define an affine scheme
$Z_{D,E}$ of finite type over $\mathbb Z[1/N]$. If it has points
over $\mathbb Z/p^r$ for arbitrarily large $r$ at one prime $p\nmid N$,
or over $\mathbb F_p$ for infinitely many primes $p\nmid N$, then
it has a point over a finite extension of $\mathbb Q$. Indeed, if its
generic fiber were empty, some nonzero integer $M$ would satisfy
$M\cdot1=0$ in its coordinate algebra. A point over $\mathbb Z/p^r$
would force $p^r\mid M$, and a point over $\mathbb F_p$ would force
$p\mid M$, contradicting either hypothesis. A closed point of the
nonempty generic fiber gives a rectifier over the asserted finite
extension. The prime-power points need not be compatible. A single point
over $\mathbb F_p$ is insufficient for this argument, and the conclusion
requires the fixed degree bounds.

\begin{proposition}\label{prop:quadratic-projection-all-fields}
Over any field $k$, every closed immersion
\[
 t\longmapsto(t^2,p(t),q(t))
\]
is rectifiable over $k$.
\end{proposition}

\begin{proof}
Write uniquely
$p(t)=p_0(t^2)+t a(t^2)$ and
$q(t)=q_0(t^2)+t b(t^2)$.
With $x=t^2$, the coordinate image is contained in the subalgebra
$k[x]\oplus t(a,b)k[x]$ of $k[t]$.  Surjectivity onto $k[t]$ forces
$(a,b)=k[x]$.  Choose $A,D\in k[x]$ with $Aa+Db=1$.  In ambient coordinates
put
\[
 y_0=y-p_0(x),\quad z_0=z-q_0(x),\quad
 T=A(x)y_0+D(x)z_0,\quad W=-b(x)y_0+a(x)z_0.
\]
Then $(T,x-T^2,W)$ is a rectifier.  Its inverse, on coordinates $(T,H,W)$,
is
\[
 x=H+T^2,\quad
 y=p_0(x)+a(x)T-D(x)W,\quad
 z=q_0(x)+b(x)T+A(x)W.
\]
The same inverse works in characteristic two.
\end{proof}

\subsection{Galois descent}

Suppose now that $L/k$ is finite Galois and that $\Phi$ is an $L$-rectifier
of a $k$-defined parameterized embedding.  Let
$\operatorname{Stab}_{\mathrm{pt}}(\mathbb A^1_L)$ denote the group of
polynomial automorphisms of $\mathbb A^3_L$ fixing $(t,0,0)$ for every $t$.
The automorphisms
\[
 c_\sigma=\Phi\circ\sigma(\Phi)^{-1}
 \qquad(\sigma\in\operatorname{Gal}(L/k))
\]
belong to this pointwise stabilizer and satisfy
$c_{\sigma\tau}=c_\sigma\,\sigma(c_\tau)$.
A rectifier of the form $s\circ\Phi$ is defined over $k$ precisely when
\[
 c_\sigma=s^{-1}\sigma(s)
 \quad\text{for every }\sigma\in\operatorname{Gal}(L/k)
\]
for a polynomial automorphism $s$ in the stabilizer. Solving the induced
normal-bundle or finite-jet equations separately does not supply such an $s$.

If $F$ vanishes on the curve and has at least two conjugates, their
product lies in the square of the curve's ideal.  Its differential
therefore vanishes along the curve, so the product cannot be an ambient
coordinate.

\subsection{Characteristic two}

\begin{question}\label{qu:other-characteristic-two}
Let $k$ be a perfect field of characteristic two, let
$i:\mathbb A^1_k\hookrightarrow\mathbb A^3_k$ be a closed embedding,
and let $f,g$ be ordered generators of its defining ideal forming
a regular sequence. Does the exact normal map induce an isomorphism
of Nisnevich sheaves
\[
 (f,g)_*:\pi_1^{\mathbb A^1}\!\left(
   \mathbb A^3_k\setminus i(\mathbb A^1_k)\right)
 \longrightarrow\pi_1^{\mathbb A^1}(\mathbb A^2_k\setminus\{0\})?
\]
\end{question}

\end{document}